\documentclass[11pt]{article}
\usepackage{a4}
\usepackage{amssymb}
\usepackage{amsmath}
\usepackage{amsthm}
\usepackage{amsfonts}
\usepackage{enumerate}
\usepackage[pagebackref,colorlinks,citecolor=blue,linkcolor=blue]{hyperref}

\usepackage{geometry}
\numberwithin{equation}{section}
\theoremstyle{plain}
\newtheorem{theorem}[equation]{Theorem}

\newtheorem{lemma}[equation]{Lemma}
\newtheorem{proposition}[equation]{Proposition}

\theoremstyle{definition}
\newtheorem{definition}[equation]{Definition}

\newtheorem{example}[equation]{Example}

\numberwithin{equation}{section}

\newcommand{\R}{{\mathbb R}}
\newcommand{\N}{{\mathbb N}}

\newcommand{\Om}{\Omega}

\providecommand{\vint}[1]{\mathchoice
          {\mathop{\vrule width 5pt height 3 pt depth -2.5pt
                  \kern -9pt \kern 1pt\intop}\nolimits_{\kern -5pt{#1}}}
          {\mathop{\vrule width 5pt height 3 pt depth -2.6pt
                  \kern -6pt \intop}\nolimits_{\kern -3pt{#1}}}
          {\mathop{\vrule width 5pt height 3 pt depth -2.6pt
                  \kern -6pt \intop}\nolimits_{\kern -3pt{#1}}}
          {\mathop{\vrule width 5pt height 3 pt depth -2.6pt
                  \kern -6pt \intop}\nolimits_{\kern -3pt{#1}}}}

\newcommand{\eps}{\varepsilon}
\newcommand{\loc}{\mathrm{loc}}

\newcommand{\BV}{\mathrm{BV}}

\newcommand{\ch}{\text{\raise 1.3pt \hbox{$\chi$}\kern-0.2pt}}

\DeclareMathOperator{\frm}{\mathfrak{m}}
\DeclareMathOperator{\dfrm}{d\mathfrak{m}}
\newcommand{\dd}{\mathrm{d}}

\DeclareMathOperator{\Mod}{Mod}
\DeclareMathOperator{\capa}{Cap}
\DeclareMathOperator{\rcapa}{cap}

\DeclareMathOperator{\diam}{diam}

\DeclareMathOperator{\supp}{spt}

\DeclareMathOperator{\inte}{int}
\DeclareMathOperator{\fint}{fine-int}

\begin{document}
\title{A topological characterization \\ of indecomposable sets of finite perimeter
\let\thefootnote\relax\footnotetext{
\hfill \break {\it Keywords\,}: set of finite perimeter, indecomposable set,
metric measure space, fine topology, two-sidedness property, Poincar\'e inequality
}}
\author{Paolo Bonicatto, Panu Lahti, and Enrico Pasqualetto}
\maketitle

\begin{abstract}
We prove that a set of finite perimeter is indecomposable if and only if it is,
up to a  choice of suitable representative, connected in the \emph{1-fine topology}.
This gives a topological characterization of indecomposability which is new even in
Euclidean spaces. Our approach relies crucially on the metric space theory of functions of
bounded variation, and we are able to prove our main result in a complete, doubling
metric measure space supporting a $1$-Poincar\'{e} inequality and having the
two-sidedness property (this class includes all
Riemannian manifolds, Carnot groups, and ${\sf RCD}(K,N)$ spaces with $K\in\mathbb R$ and $N<\infty$).
As an immediate corollary, we obtain an alternative proof of the decomposition theorem for sets of finite perimeter into maximal indecomposable components.

\end{abstract}

\section{Introduction}
A set of finite perimeter $E$ in a metric measure space $(X,d,\frm)$ is said to be \emph{indecomposable}
if it cannot be written as the disjoint union of two non-$\frm$-negligible sets $F,G$ with
$P(E,X)=P(F,X)+P(G,X)$. This measure-theoretic notion is similar to the topological notion of
connectedness. Properties of indecomposable sets in Euclidean spaces were studied by
Federer \cite[Section 4.2.25]{Fed} in the more
general framework of currents, and then by
Ambrosio et al.\ \cite{ACMM} who proved that a set of finite perimeter in $\R^n$ can always
be uniquely decomposed into maximal indecomposable sets, termed the \emph{essential connected components}.
See also \cite{ACMM} for a long list of earlier related references, in particular
in connection with image processing.

The theory was generalized to metric measure spaces by the first and third named authors
together with Tapio Rajala in \cite{BPR},
and by the second named author in \cite{L-decom}.
As is common in analysis on metric measure spaces,
in these papers as well as the current one the space $(X,d,\frm)$
is assumed to be complete, equipped with a doubling measure, and to support
a $1$-Poincar\'e inequality;
we will refer to any such space as a PI space.
We will recall the main definitions in Section \ref{sec:preliminaries}.
Indecomposable sets have been studied
also in \cite[Section 3.4]{GMS}, 
and in connection with fractional perimeter, too -- see, e.g.\, \cite{DCNRV}.
In \cite{CCDM,CCDM2} a similar concept was considered for sets that are not necessarily of finite perimeter,
and it was noted that a measure-theoretic rather than a topological notion of connectedness seems to be the 
appropriate one.

Nonetheless, throughout the development of the theory it has been a natural question whether
the measure-theoretic notion of indecomposability
would in fact coincide with the usual notion of connectedness
in a suitable topology.
In this paper we show that this is indeed the case.
See Section \ref{sec:preliminaries} for the definition of the two-sidedness
property and the $1$-fine topology.

\begin{theorem}\label{thm:iff intro}
Let $(X,d,\frm)$ be a PI space and suppose $X$ has the two-sidedness property. Let $E\subset X$ with $P(E,X)<\infty$.
Then the set $E$ is indecomposable
if and only if $\fint I_E$ is connected in the $1$-fine topology.

Moreover, the connected components of $\fint I_E$ in the $1$-fine topology are exactly
the essential connected components of $E$.
\end{theorem}

Here $I_E$ is the measure-theoretic interior, and $\fint I_E$ is its interior in the $1$-fine topology;
it turns out that the difference $I_E\setminus \fint I_E$ is $\frm$-negligible (even negligible with respect to the
codimension-one Hausdorff measure).
As a by-product of our main theorem, we obtain 
a new proof of the existence of the essential connected components of $E$.

\subsection*{Structure of the paper}
We introduce definitions and preliminary results in Section \ref{sec:preliminaries}.
In Sections \ref{sec:contruct space} and \ref{sec:locally connected}
we develop the theory of the $1$-fine topology, proving in particular that
it is locally connected.
We prove Theorem \ref{thm:iff intro} in Section \ref{sec:main result on components}.
In Section \ref{sec:fine domains} we develop further the theory of $1$-fine domains,
and in Section \ref{sec:metric} we investigate the possibility
of obtaining the essential connected components by means of the usual metric topology.
Some of our theory relies on a result of independent interest: the
(strong) Cartan property for \(p=1\), which we prove in Section \ref{sec:Cartan}.

\subsection*{Acknowledgments.}
E.\ P.\ has been supported by the Research Council of Finland grant 362898.

\section{Notation and definitions}\label{sec:preliminaries}

In this section we introduce the basic notation, definitions,
and assumptions that are employed in the paper.

\subsection{Doubling and Poincar\'e inequality}

Throughout this paper, $(X,d,\frm)$ is a complete metric space that is equip\-ped
with a metric $d$ and a Borel regular outer measure $\frm$.
To avoid certain pathological situations, we assume that $X$ consists of at least two points.
When a property holds outside a set of $\frm$-measure zero, we say that it holds
almost everywhere, abbreviated a.e.
We say that the measure $\frm$ is \textbf{doubling} if
there exists a constant $C_d\ge 1$ such that
\[
0<\frm(B(x,2r))\le C_d\frm(B(x,r))<\infty
\]
for every ball $B(x,r):=\{y\in X\colon\,d(y,x)<r\}$, with $x\in X$, $r>0$; we understand balls to be open.
We will always assume $\frm$ to be doubling.
Given a ball $B=B(x,r)$, we sometimes denote $tB:=B(x,tr)$ for \(t>0\).
By iterating the doubling condition, we obtain that for any $x\in X$ and $y\in B(x,R)$ with $0<r\le R<\infty$, we have
\begin{equation}\label{eq:homogenous dimension}
	\frac{\frm(B(y,r))}{\frm(B(x,R))}\ge \frac{1}{C_d^2}\left(\frac{r}{R}\right)^{Q},
\end{equation}
where $Q>1$ only depends on the doubling constant $C_d$.

From the completeness of the space and the doubling property of the measure $\frm$,
we obtain that $X$ is proper, meaning that closed and bounded sets are compact.
All functions defined on $X$ or its subsets will take values in $[-\infty,\infty]$.
Given an open set $W\subset X$, we define $L^1_{\loc}(W)$
to be the class of functions $u$ on $W$ such that $u\in L^1(W')$ for every open $W'\Subset W$, where the latter notation means that
$\overline{W'}$ is a compact subset of $W$.
Other local spaces of functions are defined analogously.

\subsection{Curves, (weak) upper gradients, Newton-Sobolev space}
By a \textbf{(rectifiable) curve} we mean a continuous mapping $\gamma \colon I \to X$ from a compact interval $I$ of the real line into $X$, with finite \textbf{length} $\ell_{\gamma}<\infty$.
We assume every curve to be parametrized
by arc-length, which can always be done (see e.g.\ \cite[Theorem~3.2]{Hj}).

A nonnegative Borel function $g$ on $X$ is an \textbf{upper gradient
of a function} $u$
on $X$ if for all nonconstant curves $\gamma\colon [0,\ell_{\gamma}]\to X$, we have
\begin{equation}\label{eq:definition of upper gradient}
|u(x)-u(y)|\le \int_{\gamma} g\,ds:=\int_0^{\ell_{\gamma}} g(\gamma(s))\,\dd s,
\end{equation}
where $x$ and $y$ are the end points of $\gamma$.
We interpret $|u(x)-u(y)|=\infty$ whenever  
at least one of $|u(x)|$, $|u(y)|$ is infinite.
Upper gradients were originally introduced in \cite{HK}.

The $1$-\textbf{modulus of a family of curves} $\Gamma$ is defined by
\[
\Mod_{1}(\Gamma):=\inf\int_{X}\rho\, \dfrm,
\]
where the infimum is taken over all nonnegative Borel functions $\rho$
such that $\int_{\gamma}\rho\,\dd s\ge 1$ for every curve $\gamma\in\Gamma$.
A property is said to hold for $1$-almost every curve
if it fails only for a curve family with zero $1$-modulus. 
If $g$ is a nonnegative $\frm$-measurable function on $X$
and (\ref{eq:definition of upper gradient}) holds for $1$-almost every curve,
we say that $g$ is a \textbf{$1$-weak upper gradient} of $u$. 
By only considering curves $\gamma\colon [0,\ell_{\gamma}]\to A\subset X$,
we can talk about a function $g$ being a ($1$-weak) upper gradient of $u$ in $A$.

We say that $X$ supports a $1$-Poincar\'e inequality
if there exist finite constants $C_P>0$ and $\lambda \ge 1$ such that for every
ball $B(x,r)$, every $u\in L^1_{\loc}(X)$,
and every upper gradient $g$ of $u$,
we have
\[
\vint{B(x,r)}|u-u_{B(x,r)}|\, \dfrm 
\le C_P r\vint{B(x,\lambda r)}g\dfrm,
\]
where 
\[
u_{B(x,r)}:=\vint{B(x,r)}u\dfrm :=\frac 1{\frm(B(x,r))}\int_{B(x,r)}u\dfrm.
\]
We will always assume $X$ to support a $1$-Poincar\'e inequality.
A complete metric space with a doubling measure and supporting a Poincar\'e inequality is commonly called a PI space.

Let $W\subset X$ be an open set. We let
\[
\Vert u\Vert_{N^{1,1}(W)}:=\Vert u\Vert_{L^1(W)}+\inf \Vert g\Vert_{L^1(W)},
\]
where the infimum is taken over all $1$-weak upper gradients $g$ of $u$ in $W$.
Then we define the \textbf{Newton-Sobolev space}
\[
N^{1,1}(W):=\{u\colon \|u\|_{N^{1,1}(W)}<\infty\},
\]
which was first introduced in \cite{S}.

We understand Newton-Sobolev functions to be defined at every $x\in W$
(even though $\Vert \cdot\Vert_{N^{1,1}(W)}$ is then only a seminorm).
It is known that for any $u\in N_{\loc}^{1,1}(W)$ there exists a minimal $1$-weak
upper gradient of $u$ in $W$, always denoted by $g_{u}$, satisfying $g_{u}\le g$ 
a.e.\ in $W$ for every $1$-weak upper gradient $g\in L_{\loc}^{1}(W)$
of $u$ in $W$, see e.g.\ \cite[Theorem 2.25]{BB}.

\subsection{Functions of bounded variation}

Next we present the definition and basic properties of functions
of bounded variation on metric spaces, following \cite{M}. See also e.g.\ \cite{AFP, Fed} for the classical 
theory in the Euclidean setting.
Given an open set $W\subset X$ and a function $u\in L^1_{\loc}(W)$,
we define the \textbf{total variation} of $u$ in $W$ by
\[
\|Du\|(W):=\inf\left\{\liminf_{i\to\infty}\int_W g_{u_i}\dfrm:\, u_i\in N^{1,1}_{\loc}(W),\, u_i\to u\textrm{ in } L^1_{\loc}(W)\right\},
\]
where each $g_{u_i}$ is again the minimal $1$-weak upper gradient of $u_i$ in $W$.
(In \cite{M}, pointwise Lipschitz constants were used in place of upper gradients, but the theory
can be developed similarly with either definition.)
We say that a function $u\in L^1(W)$ is of bounded variation, 
and denote $u\in\BV(W)$, if $\|Du\|(W)<\infty$.

For an arbitrary set $A\subset X$, we define
\[
\|Du\|(A):=\inf\{\|Du\|(W):\, A\subset W,\,W\subset X
\text{ is open}\}.
\]
If $u\in L^1_{\loc}(W)$ and $\Vert Du\Vert(W)<\infty$,
then $\|Du\|(\cdot)$ is
a Borel regular outer measure on $W$ by \cite[Theorem 3.4]{M}.
An $\frm$-measurable set $E\subset X$ is said to have \textbf{finite perimeter}
in $W$ if $\|D\ch_E\|(W)<\infty$, where $\ch_E\colon X\to \{0,1\}$ is the characteristic function of $E$.
The perimeter of $E$ in $W$ is also denoted by
\[
P(E,W):=\|D\ch_E\|(W).
\]
If $P(E,X)<\infty$, we say briefly that $E$ is a set of finite perimeter.

For any measurable sets
$E_1,E_2\subset X$ and an open set $W\subset X$, we have
\begin{equation}\label{eq:perimeter subadditivity}
P(E_1\cap E_2,W)+P(E_1\cup E_2,W)\le P(E_1,W)+P(E_2,W);
\end{equation}
see \cite[Proposition 4.7]{M}; for sets $E_1,E_2$ with finite perimeter in $W$ the same then holds also with $W$
replaced by an arbitrary $A\subset W$.
Using this inequality (only for the union) as well as the lower semicontinuity of the total
variation with respect to $L_{\loc}^1$-convergence in open sets, we have
for any open $W\subset X$ and any $E_1,E_2\ldots \subset X$ that
\begin{equation}\label{eq:perimeter of union}
P\Bigg(\bigcup_{j=1}^{\infty}E_j,W\Bigg)
\le \sum_{j=1}^{\infty}P(E_j,W).
\end{equation}

The following \textbf{coarea formula} is given in \cite[Proposition 4.2]{M}:
if $\Omega\subset X$ is an open set and $u\in L^1_{\loc}(\Omega)$, then
\begin{equation}\label{eq:coarea}
\|Du\|(\Omega)=\int_{-\infty}^{\infty}P(\{u>t\},\Omega)\,\dd t,
\end{equation}
where we abbreviate $\{u>t\}:=\{x\in \Om:\,u(x)>t\}$.

Applying the Poincar\'e inequality to sequences of approximating
$N^{1,1}_{\loc}$-functions in the definition of the total variation, we get
the following $\BV$ version:
for every ball $B(x,r)$ and every 
$u\in L^1_{\loc}(X)$, we have
\[
\int_{B(x,r)}|u-u_{B(x,r)}|\dfrm
\le C_P r \Vert Du\Vert (B(x, \lambda r)).
\]
The $1$-Poincar\'e inequality implies the so-called Sobolev-Poincar\'e inequality,
see e.g.\ \cite[Theorem 4.21]{BB}, and by applying the latter to approximating $N^{1,1}_{\loc}$-functions
in the definition of the total variation, we get the following \textbf{Sobolev-Poincar\'e inequality} for $\BV$ functions:
For every ball $B(x,r)$ and every $u\in L^1_{\loc}(X)$, we have
\[
\left(\,\vint{B(x,r)}|u-u_{B(x,r)}|^{Q/(Q-1)}\,\dfrm\right)^{(Q-1)/Q}
\le C_{SP}r\frac{\Vert Du\Vert (B(x,2\lambda r))}{\frm(B(x,2\lambda r))},
\]
where $Q$ is the exponent from \eqref{eq:homogenous dimension} and
$C_{SP}=C_{SP}(C_d,C_P,\lambda)\ge 1$ is a finite constant.
For an $\frm$-measurable set $E\subset X$, this implies (see e.g.\ \cite[Equation (3.1)]{KoLa})
\[
\frac 12 \left(\frac{\min\{\frm(B(x,r)\cap E),\frm(B(x,r)\setminus E)\}}{\frm(B(x,r))}\right)^{(Q-1)/Q}
\le C_{SP}r\frac{P(E,B(x,2\lambda r))}{\frm(B(x,2\lambda r))}.
\]
Rearranged, this implies
\begin{equation}\label{eq:relative isoperimetric inequality}
	\begin{split}
		&\min\{\frm(B(x,r)\cap E),\frm(B(x,r)\setminus E)\}\\
		&\quad \le 2 C_{SP}r\left(\frac{\min\{\frm(B(x,r)\cap E),\frm(B(x,r)\setminus E)\}}{\frm(B(x,r))}\right)^{1/Q}P(E,B(x,2\lambda r)).
	\end{split}
\end{equation}
From the $1$-Poincar\'e inequality,
by \cite[Theorem 4.21, Theorem 5.51]{BB}
we also get the following Sobolev inequality:
if $x\in X$, $0<r<\frac{1}{4}\diam X$, and $u\in N^{1,1}(X)$ with $u=0$
in $X\setminus B(x,r)$, then
\[
\int_{B(x,r)} |u|\dfrm \le C_S r \int_{B(x,r)}  g_u\dfrm
\]
for a finite constant $C_S=C_S(C_d,C_P)\ge 1$.
For any $\frm$-measurable set $E\subset B(x,r)$, applying the Sobolev inequality
to a suitable sequence approximating $\ch_E$,
we get the \textbf{isoperimetric inequality}
\begin{equation}\label{eq:isop inequality with zero boundary values}
\frm(E)\le C_S r P(E,X).
\end{equation}

\subsection{Lower and upper approximate limits, Hausdorff measure and the measure topology}
The \textbf{lower} and \textbf{upper approximate limits} of an extended real-valued function $u$ on $X$ are defined respectively by
\begin{equation}\label{eq:lower approximate limit}
	u^{\wedge}(x):
	=\sup\left\{t\in\R:\,\lim_{r\to 0}\frac{\frm(B(x,r)\cap\{u<t\})}{\frm(B(x,r))}=0\right\}
\end{equation}
and
\begin{equation}\label{eq:upper approximate limit}
	u^{\vee}(x):
	=\inf\left\{t\in\R:\,\lim_{r\to 0}\frac{\frm(B(x,r)\cap\{u>t\})}{\frm(B(x,r))}=0\right\}.
\end{equation}

For any set $A\subset X$ and $0<R<\infty$, the restricted \textbf{Hausdorff content
of codimension one} is defined by
\[
\mathcal{H}_{R}(A):=\inf\left\{ \sum_{j\in I}
\frac{\frm(B(x_{j},r_{j}))}{r_{j}}:\,A\subset\bigcup_{j\in I}B(x_{j},r_{j}),\,r_{j}\le R\right\},
\]
where the infimum is taken over at most countable index sets $I$.
The \textbf{codimension-one Hausdorff measure} of $A\subset X$ is then defined by
\[
\mathcal{H}(A):=\lim_{R\rightarrow 0}\mathcal{H}_{R}(A).
\]

For $E\subset X$ and $x\in X$, we denote
\begin{equation}\label{eq:upper measure density}
\Theta^*(E,x):=\limsup_{r\to 0}\frac{\frm(E\cap B(x,r))}{\frm(B(x,r))},
\end{equation}
and if the limit exists, we denote it by $\Theta(E,x)$.
For any $E\subset X$, the \textbf{measure-theoretic interior} $I_E$ is defined as the set of points
$x\in X$ for which $\Theta(X\setminus E,x)=0$, whereas the
\textbf{measure-theoretic exterior} $O_E$ is defined as the set of points
$x\in X$ for which $\Theta(E,x)=0$.
The \textbf{measure-theoretic boundary} $\partial^{*}E$ is defined
as the set of points $x\in X$
at which both $E$ and its complement have strictly positive upper density, i.e.\
\[
\Theta^*(E,x)>0\quad
\textrm{and}\quad\Theta^*(X\setminus E,x)>0.
\]

We define the \textbf{measure topology} as follows.

\begin{definition}\label{def:measure topology}
A set $U\subset X$ is open in the measure topology if $\Theta(X\setminus U,x)=0$ for every $x\in U$.
\end{definition}

Given a number $0<\gamma\le 1/2$, we also define the \textbf{strong boundary}
\begin{equation}\label{eq:strong boundary}
	\Sigma_{\gamma} E:=\left\{x\in X:\, \liminf_{r\to 0}\frac{\frm(B(x,r)\cap E)}{\frm(B(x,r))}\ge \gamma\ \, \textrm{and}\ \, \liminf_{r\to 0}\frac{\frm(B(x,r)\setminus E)}{\frm(B(x,r))}\ge \gamma\right\}.
\end{equation}
For an open set $\Omega\subset X$ and an $\frm$-measurable set $E\subset X$ with
$P(E,\Omega)<\infty$, we have
\begin{equation}\label{eq:strong boundary and measure}
\mathcal H((\partial^*E\setminus \Sigma_{\gamma}E)\cap\Om)=0
\end{equation}
for $\gamma \in (0,1/2]$ that only depends on $C_d$ and $C_P$,
see \cite[Theorem 5.4]{A1}.
Furthermore, for any Borel set $A\subset \Om$ we have
\begin{equation}\label{eq:def of theta}
P(E,A)=\int_{\partial^{*}E\cap A}\theta_E\,\dd\mathcal H,
\end{equation}
where
$\theta_E\colon \partial^*E\cap \Om\to [\alpha,C_d]$
with $\alpha=\alpha(C_d,C_P,\lambda)>0$, see \cite[Theorem 5.3]{A1} 
and \cite[Theorem 4.6]{AMP}.

In particular, $P(E,\Om)<\infty$ implies $\mathcal H(\partial^*E\cap \Om)<\infty$.
By Federer's characterization of sets of finite perimeter,
the converse also holds
(see \cite[Section 4.5.11]{Fed} for the original Euclidean result).

\begin{theorem}[{\cite[Theorem 1.1]{L-Fedchar}}]\label{thm:Federers characterization}
Let $\Om\subset X$ be an open set, let $E\subset X$ be an $\frm$-measurable set, and
suppose that $\mathcal H(\partial^*E\cap \Om)<\infty$. Then $P(E,\Om)<\infty$.
\end{theorem}

\begin{definition}\label{def:isotropicity}
We say that $(X,d,\frm)$ is \textbf{isotropic} if for
any pair of sets $E,F\subset X$ of finite perimeter with $F\subset E$
it holds that
\[
\theta_F(x) = \theta_E(x)\quad \textrm{for }\mathcal H\textrm{-a.e.\ }x\in \partial^*F\cap\partial^*E.
\]
\end{definition}

\begin{definition}\label{def:two sided}
	We say that the space $(X,d,\frm)$ has the \textbf{two-sidedness property} if for any two disjoint sets of finite 
	perimeter $E,F\subset X$, we have for $\mathcal H$-a.e.\ $x\in \partial^*E\cap \partial^*F$ that
	\[
	\lim_{r\to 0}\frac{\frm(B(x,r)\setminus (E\cup F))}{\frm(B(x,r))}=0.
	\]
\end{definition}

By \cite[Proposition 6.2]{AMP}
the two-sidedness property implies isotropicity.
Besides all Euclidean spaces (and, more generally, all smooth Riemannian manifolds), the class of PI spaces
with the two-sidedness property includes all \({\sf RCD}(K,N)\) spaces with \(K\in\R\) and \(N\in[1,\infty)\),
as was pointed out in \cite[Example 1.31]{BPR}. In the next result, we prove that every Carnot group
has the two-sidedness property.
\begin{proposition}
Every Carnot group \(\mathbb G\) has the two-sidedness property.
\end{proposition}
\begin{proof}
The notation we will use in this proof is borrowed from \cite{AS}. Fix two disjoint sets of finite perimeter \(E,F\subset\mathbb G\).
Applying \cite[Corollary 2.5]{AS}, we obtain that
\[\begin{split}
\nu_E(x)=\nu_{\mathbb G\setminus F}(x)=-\nu_F(x)&\quad\text{ for }\mathcal H\text{-a.e.\ }x\in\mathcal F E\cap\mathcal F(\mathbb G\setminus F)=\mathcal F E\cap\mathcal F F,\\
\nu_{E\cup F}(x)=\nu_E(x)&\quad\text{ for }\mathcal H\text{-a.e.\ }x\in\mathcal F(E\cup F)\cap\mathcal F E,\\
\nu_{E\cup F}(x)=\nu_F(x)&\quad\text{ for }\mathcal H\text{-a.e.\ }x\in\mathcal F(E\cup F)\cap\mathcal F F,
\end{split}\]
whence it follows that
\[
\nu_E(x)=-\nu_F(x)=-\nu_{E\cup F}(x)=-\nu_E(x)\quad\text{ for }\mathcal H\text{-a.e.\ }x\in\mathcal F E\cap\mathcal F F\cap\mathcal F(E\cup F).
\]
Since \(|\nu_E(x)|=1\) for every \(x\in\mathcal F E\) and \(\mathcal F E\cap\mathcal F F\cap\mathcal F(E\cup F)\) is \(\mathcal H\)-a.e.\ equivalent
to \(\partial^*E\cap\partial^*F\cap\partial^*(E\cup F)\), we conclude that \(\mathcal H(\partial^*E\cap\partial^*F\cap\partial^*(E\cup F))=0\), which
shows that \(\mathbb G\) has the two-sidedness property.
\end{proof}

It would be interesting to understand whether all \({\sf CD}(K,N)\) spaces (or even \({\sf MCP}(K,N)\) spaces),
with \(K\in\R\) and \(N\in[1,\infty)\), have the two-sidedness property.
\subsection{Capacities}

The \textbf{$1$-capacity} of a set $A\subset X$ is defined by
\[
\capa_1(A):=\inf \Vert u\Vert_{N^{1,1}(X)},
\]
where the infimum is taken over all functions $u\in N^{1,1}(X)$ satisfying
$u\ge 1$ in $A$.
The \textbf{variational $1$-capacity} of a set $A\subset D$
with respect to a set $D\subset X$ is defined by
\[
\rcapa_1(A,D):=\inf \int_X g_u \dfrm,
\]
where the infimum is taken over functions $u\in N^{1,1}(X)$ satisfying
$u=0$ in $X\setminus D$ and
$u\ge 1$ in $A$, and $g_u$ is the minimal $1$-weak upper gradient of $u$ (in $X$).
By truncation, we see that we can assume $0\le u\le 1$ on $X$.
For basic properties satisfied by capacities, such as monotonicity and countable subadditivity, see e.g.\ \cite{BB}.

A fact that is often useful for us is that $\rcapa_1$ is an outer capacity, meaning
that for $A\Subset \Om$ with $\Om\subset X$ open, we have
\begin{equation}\label{eq:outer capacity}
\rcapa_1(A,\Om)=\inf_{\substack{A\subset V\subset \Om\\ V\textrm{ is open } }}\rcapa_1(V,\Om),
\end{equation}
see \cite[Proposition 6.19(vii)]{BB}.

By \cite[Theorem 4.3, Theorem 5.1]{HaKi} we know that for $A\subset X$,
\begin{equation}\label{eq:null sets of Hausdorff measure and capacity}
	\capa_1(A)=0\quad \textrm{if and only if}\quad \mathcal H(A)=0.
\end{equation}
For any set $A\subset B(x,r)$, we have
\begin{equation}\label{eq:cap and H measure}
\rcapa_1(A,2B)\le C_d \mathcal H(A);
\end{equation}
this can be proved by considering cutoff functions corresponding
to coverings $\{B_j\}_j$ of $A$; see e.g.\ the proof of \cite[Lemma 3.3]{KKST}.

By \cite[Lemma 11.22]{BB}, for $A\subset B=B(x,r)$ and $t>2$ with $tr<\tfrac 14 \diam X$, we have
\begin{equation}\label{eq:capacities}
	\rcapa_1(A,2B)\le C'(1+t)\rcapa_1(A,tB)
\end{equation}
for a finite constant $C'$ depending only on the doubling and Poincar\'e constants.

By \cite[Proposition 6.16]{BB} we know that for $x\in X$ and $0<r<\tfrac 18 \diam X$, we have
\begin{equation}\label{eq:cap estimate}
\frac{\frm(B(x,r))}{C''r} 
\le \rcapa_1(B(x,r),B(x,2r))
\le C_d\frac{\frm(B(x,r))}{r} 
\end{equation}
for some finite constant $C''$  depending only on the doubling and Poincar\'e constants.
By \cite[Proposition 6.16]{BB}, in this case for $A\subset B(x,r)$ we also have
\begin{equation}\label{eq:Cap and cap estimate}
\rcapa_1(A,B(x,2r))\le 2(1+r^{-1})\capa_1(A).
\end{equation}

We also define a $\BV$-capacity of $A\subset X$ by
\[
\capa_{\BV}(A):=\inf (\Vert u\Vert_{L^1(X)}+\Vert Du\Vert(X)),
\]
where the infimum is taken over functions $u\in \BV(X)$ with $u\ge 1$ in a neighborhood of $A$.
By \cite[Theorem 4.3]{HaKi}, we have
\begin{equation}\label{eq:cap and BV cap}
\capa_1(A)\le C_r'\capa_{\BV}(A)
\end{equation}
for every $A\subset X$ and for a finite constant
$C'_{\textrm{r}}=C'_{\textrm{r}}(C_d,C_P,\lambda )\ge 1$.

We also define a \textbf{variational $\BV$-capacity} for any $A\subset\Om$, with
$\Om\subset X$ open, by
\[
\rcapa^{\vee}_{\BV}(A,\Om):=\inf \Vert Du\Vert(X),
\]
where the infimum is taken over functions $u\in \BV(X)$ such that
$u^{\wedge}=u^{\vee}= 0$ $\mathcal H$-a.e.\ in $X\setminus \Om$ and
$u^{\vee}\ge 1$ $\mathcal H$-a.e.\ in $A$.
By \cite[Theorem 5.7]{L-SS} we know that
\begin{equation}\label{eq:variational one and BV capacity}
	\rcapa_{1}(A,\Om)\le C_{\textrm{r}}\rcapa^{\vee}_{\BV}(A,\Om)
\end{equation}
for a finite constant $C_{\textrm{r}}=C_{\textrm{r}}(C_d,C_P,\lambda )\ge 1$.

\subsection{The \texorpdfstring{$1$}{1}-fine topology}
We begin with the following definition. 
\begin{definition}\label{def:1 fine topology}
	We say that $A\subset X$ is \textbf{$1$-thin} at the point $x\in X$ if
	\[
	\lim_{r\to 0}\frac{\rcapa_1(A\cap B(x,r),B(x,2r))}{\rcapa_1(B(x,r),B(x,2r))}=0.
	\]
	We also say that a set $U\subset X$ is \textbf{$1$-finely open} if $X\setminus U$ is $1$-thin at every $x\in U$. Then we define the \textbf{$1$-fine topology} as the collection of $1$-finely open sets on $X$.
	
	We denote the \textbf{$1$-fine interior} of a set $H\subset X$, i.e.\ the largest $1$-finely open set contained in $H$, by $\fint H$. We denote the \textbf{$1$-fine closure} of $H$,
	i.e.\ the smallest $1$-finely closed set containing $H$, by $\overline{H}^1$. The $1$-fine boundary of $H$
	is $\partial^1 H:=\overline{H}^1\setminus \fint H$.
    The $1$-base $b_1 H$ is the set of points where $H$ is not $1$-thin.
\end{definition}

See \cite[Section 4]{L-FC} for discussion on this definition, and for a proof of the fact that the $1$-fine topology is indeed a topology.
Now we collect some facts on the $1$-fine topology, and its connection with sets of finite perimeter.

Let $A\subset X$.
By \cite[Lemma 3.1]{L-Fed}, $1$-thinness implies zero measure density, i.e.\
\begin{equation}\label{eq:thinness and measure thinness}
	\textrm{if }A\textrm{ is 1-thin at }x,\textrm{ then }\Theta(A,x)=0.
\end{equation}
By \cite[Corollary 3.5]{L-Fed}, we have
\begin{equation}\label{eq:fine closure}
	\overline{A}^1=A\cup b_1 A,
\end{equation}
so that also
\begin{equation}\label{eq:fine int}
	\fint{A}=A\setminus b_1 (X\setminus A),
\end{equation}
and also
\begin{equation}\label{eq:fine bdry char}
	\partial^ 1 A = (A \cap b_1(X \setminus A)) \cup ((X \setminus A) \cap b_1 A).
\end{equation}

By \cite[Proposition 3.3]{L-Fed}, for any $A\subset D\subset X$ we have
\begin{equation}\label{eq:cap of fine closure}
\capa_1(\overline{A}^1)=\capa_1(A)
\quad\textrm{and}\quad
\rcapa_1(A,D) =\rcapa_1(\overline{A}^1\cap D,D).
\end{equation}

\begin{definition}\label{def:quasiopen}
A set $U\subset X$ is \textbf{$1$-quasiopen} if for every $\eps>0$ there exists an open set $G\subset X$
such that $U\cup G$ is open and $\capa_1(G)<\eps$.
\end{definition}

By \cite[Corollary 6.12]{L-CK}
we know that for an arbitrary set $U\subset X$,
\begin{equation}\label{eq:finely and quasiopen}
U\textrm{ is }1\textrm{-quasiopen }\ \Longleftrightarrow\ \ U=V\cup N\
\textrm{ where }V\textrm{ is 1-finely open and }\mathcal H(N)=0.
\end{equation}

\begin{lemma}[{\cite[Lemma 4.4]{L-WC}}]\label{lem:open thin set}
If $A\subset X$ is $1$-thin at $x\in X\setminus A$, then there exists an
open set $W\supset A$ that is also $1$-thin at $x$.
\end{lemma}

\begin{proposition}[{\cite[Proposition 4.5]{L-WC}}]\label{eq:capa x positive}
Let $x\in X$ with $\capa_1(\{x\})>0$. Then $\{x\}$ is not $1$-thin at $x$.
\end{proposition}

Recall that $I_E$ is the measure-theoretic interior of $E$, i.e.\ the set where
$\Theta(X\setminus E,x)=0$. From \eqref{eq:thinness and measure thinness} and
\eqref{eq:fine bdry char}, it follows that always $\partial^*E\subset \partial^1 I_E$.

\begin{theorem}[{\cite[Theorem 1.1]{L-Fed}}]\label{thm:Federer fine}
	For an open set $\Om\subset X$ and an $\frm$-measurable set $E\subset X$, we have
	$P(E,\Om)<\infty$ if and only if $\mathcal H(\partial^1 I_E\cap \Om)<\infty$.
	Furthermore, then
	\begin{equation}\label{eq:fine boundary and measure}
	\mathcal H((\partial^1 I_E \setminus\partial^*E) \cap \Om) = 0.
	\end{equation}
\end{theorem}

We also have the following.
\begin{proposition}[{\cite[Corollary 4.4]{L-Fed}}]\label{prop:IE and OE}
	If  $\Om \subset X$ is open
	and $E\subset X$ is such that
	$P(E,\Om)<\infty$, then each of the sets $I_E \cap \Om$
	and $O_E \cap  \Om$ is the union of a $1$-finely open set
	and an $\mathcal H$-negligible set.
\end{proposition}

\begin{lemma}\label{lem:cap negligible}
Let $D,H\subset X$ with $\mathcal H(D\Delta H)=0$. Then
\[
\fint D \Delta\fint H,
\quad \overline{D}^1 \Delta\overline{H}^1,
\quad \partial^1 D\Delta\partial^1 H
\]
are all $\mathcal H$-negligible.
\end{lemma}

Recall from \eqref{eq:null sets of Hausdorff measure and capacity} that being
$\mathcal H$-negligible is equivalent to being
$\capa_1$-negligible.

\begin{proof}
By \eqref{eq:fine closure}, we have $\overline{D}^1=D\cup b_1 D$ and $\overline{H}^1=H\cup b_1 H$.
From \eqref{eq:cap and H measure} we get
$b_1 D= b_1 H$. Thus $\mathcal H(\overline{D}^1 \Delta\overline{H}^1)=0$.

Since $\fint D=X\setminus \overline{X\setminus D}^1$, we also get
$\mathcal H(\fint{D} \Delta\fint{H})=0$.

Since $\partial^1 D= \overline{D}^1\setminus \fint D$, we finally get
$\mathcal H(\partial^1 D \Delta\partial^1 H)=0$.
\end{proof}

\subsection{Indecomposable sets}\label{sec:indecomposable}
\begin{definition}
	Let $E\subset X$ be a set of finite perimeter. Given any Borel set $D\subset X$, we say that
	$E$ is \textbf{decomposable} in $D$ provided there exists a partition $\{F,G\}$ of $E\cap D$ into sets of finite perimeter such that $\frm(F),\frm(G)>0$ and $P(E,D)=P(F,D)+P(G,D)$.
	On the other hand, we say that $E$ is \textbf{indecomposable in} $D$ if it is not
	decomposable in $D$. For brevity, we say that $E$ is \textbf{decomposable}
	(resp. \textbf{indecomposable}) provided it is decomposable in $X$ (resp. indecomposable in $X$).
\end{definition}

We record some results from \cite{BPR} that we will need later. Recall
the definition of isotropic space from Definition \ref{def:isotropicity}.

\begin{lemma}[{\cite[Lemma 2.3]{BPR}}]\label{lem:measure theoretic boundaries}
Suppose $X$ is isotropic. Let $E,F\subset X$ be sets of finite perimeter.
Then the following hold:
\begin{enumerate}[(i)]
\item If $P(E\cup F,X)=P(E,X)+P(F,X)$, then $\mathcal H(\partial^*E\cap \partial^*F)=0$.
\item Suppose $E,F\subset X$ are disjoint.
If $\mathcal H(\partial^*E\cap \partial^*F)=0$, then
\[
P(E\cup F,X)=P(E,X)+P(F,X).
\]
\end{enumerate}
\end{lemma}

\begin{theorem}[{\cite[Theorem 2.14]{BPR}} and {\cite[Theorem 1.1]{L-decom}}]\label{thm:essential connected components}
Given a set of finite perimeter $E\subset X$, there exists a unique
(modulo sets of $\frm$-measure zero)
finite or countable partition $\{E_j\}_{j}$ of $E$ into indecomposable sets $E_j$ such that
$\frm(E_j)>0$ for every $j$, and 
\[
P(E,X)=\sum_{j}P(E_j,X).
\]
\end{theorem}
The sets $E_j$ are called the \emph{essential connected components} of $E$.

We say that $X$ is geodesic if for every pair of points $x,y\in X$, there exists a curve
$\gamma\colon [0,\ell_{\gamma}]\to X$ with $\gamma(0)=x$, $\gamma(\ell_{\gamma})=y$, and $\ell_{\gamma}=d(x,y)$.

We have the following connection between essential connected components and connected components
in the metric topology.

\begin{proposition}[{\cite[Proposition 4.1]{L-Ftype}}]\label{prop:connected components}
	Suppose $X$ is geodesic.
	Let $\Om\subset X$ be open and let
	$F\subset \Om$ be relatively closed with $\ch_F\in \BV_{\loc}(\Om)$.
	Denote the components of $F$ having nonzero
	$\frm$-measure by $F_1,F_2,\ldots$. Then
	\[
	\frm\Bigg(F\setminus \bigcup_{j}F_j\Bigg)=0
	\quad\ \textrm{and}\quad\
	P(F,A)=\sum_{j}P(F_j,A)\ \ \textrm{for any }A\subset\Om.
	\]
\end{proposition}

\subsection{Motivating examples}

A natural question is whether one could use
the metric topology or some other natural
topology in formulating a result similar to Theorem 
\ref{thm:iff intro}. We consider some simple examples.

\begin{example}\label{ex:rationals}
	Consider the (unweighted) plane $\R^2$ and the set of finite perimeter
	\[
	E:=\R^2\setminus \bigcup_{j=1}^{\infty}B(q_j,2^{-j}),
	\]
	where $\{q_j\}_{j=1}^{\infty}$ is an enumeration of the points on the plane with rational coordinates.
	If we use the metric topology, the interior $\inte I_E$ is empty, and so Theorem \ref{thm:iff intro}
	cannot be formulated in the metric topology.
	The same problem arises if we consider any other representative of $E$ in place of $I_E$.
	
	On the other hand, if in place of $\inte I_E$
	one considers, say, $I_E$ or $\overline{I_E}$ (and still considers the metric topology),
    one also runs into problems
    with certain sets $E$, as we will see
	in the next example.
\end{example}

\begin{example}\label{ex:disks}
	Consider the (unweighted) plane $\R^2$ and the set
	\[
	E:=B((0,0),1)\cup B((2,0),1).
	\]
	Now the essential connected components of $E$ are the two disks $B((0,0),1)$  and  $B((2,0),1)$.
	
	On the other hand, $I_E=E\cup \{(1,0)\}$, and then $\fint I_E=E$, so that the connected components of $\fint I_E$
	in the $1$-fine topology are also $B((0,0),1)$  and  $B((2,0),1)$,
    consistently with Theorem \ref{thm:iff intro}.
	
	In this case, the connected components of $\inte I_E=E$ in the metric topology
    \emph{are} also the two disks,
	but as we saw in the previous example, considering $\inte I_E$ does not always work.
	On the other hand, if we consider $I_E$ or $\overline{I_E}$, these are
    equal to $E\cup \{(1,0)\}$ and $\overline{E}$ respectively,
    which are both connected in the metric topology.

	Alternatively, one might consider the measure topology
    in which a set $A$ is open if $\Theta(X\setminus A,x)=0$
	for all $x\in A$ (as in Definition \ref{def:measure topology}).
	However, in this topology we would have that $I_E$ is itself a connected open set.
\end{example}

These examples motivate us to consider the $1$-fine topology.

\subsection{Superminimizers}\label{sec:superminimizer}

In this section we record and prove some results on superminimizers.

In this subsection, the symbol $\Omega$ will always denote a nonempty open subset of $X$.
We denote by $\BV_c(\Om)$ the class of functions $\varphi\in\BV(\Om)$ with compact
support in $\Om$, that is, $\supp \varphi\Subset \Om$.

\begin{definition}
	We say that $u\in\BV_{\loc}(\Om)$ is a \textbf{$1$-minimizer} in $\Om$ if
	for all $\varphi\in \BV_c(\Om)$,
	\begin{equation}\label{eq:definition of 1minimizer}
		\Vert Du\Vert(\supp\varphi)\le \Vert D(u+\varphi)\Vert(\supp\varphi).
	\end{equation}
	We say that $u\in\BV_{\loc}(\Om)$ is a \textbf{$1$-superminimizer} in $\Om$
	if \eqref{eq:definition of 1minimizer} holds for all nonnegative $\varphi\in \BV_c(\Om)$.
	We say that $u\in\BV_{\loc}(\Om)$ is a \textbf{$1$-subminimizer} in $\Om$ if
	\eqref{eq:definition of 1minimizer} holds for all nonpositive $\varphi\in \BV_c(\Om)$,
	or equivalently if $-u$ is a $1$-superminimizer in $\Om$.
\end{definition}

Given a function $\psi\colon\Om\to\overline{\R}$, with $\Om$ bounded,
and $f\in L^1_{\loc}(X)$
with $\Vert Df\Vert(X)<\infty$, we define the class of admissible functions
\[
\mathcal K_{\psi,f}(\Om):=\{u\in\BV_{\loc}(X):\,u\ge \psi\textrm{ in }\Om\textrm{ and }u=f\textrm{ in }X\setminus\Om\}.
\]
The (in)equalities above are understood in the a.e.\ sense, since $\BV$ functions are only
defined up to sets of $\frm$-measure zero.
Clearly $\Vert Du\Vert(X)<\infty$
for every $u\in\mathcal K_{\psi,f}(\Om)$.

\begin{definition}
	We say that $u\in\mathcal K_{\psi,f}(\Om)$ is a solution of the $\mathcal K_{\psi,f}(\Om)$-obstacle problem
	if $\Vert Du\Vert(X)\le \Vert Dv\Vert(X)$ for all $v\in\mathcal K_{\psi,f}(\Om)$.
    If $u=\ch_E$ for some $E\subset X$ and $u$ is a solution, we also call the set $E$ a solution.
\end{definition}

The following fact follows directly from the definitions.

\begin{proposition}\label{prop:solutions are superminimizers}
	If $u\in\mathcal K_{\psi,f}(\Om)$ is a solution
	of the $\mathcal K_{\psi,f}(\Om)$-obstacle problem, then $u$
	is a $1$-superminimizer in $\Om$.
\end{proposition}

Recall the lower and upper approximate limits $u^\wedge$ and $u^\vee$ from \eqref{eq:lower approximate limit}
and \eqref{eq:upper approximate limit}.

\begin{theorem}[{\cite[Theorem 3.11]{L-WC}}]\label{thm:superminimizers are lsc}
	Let $u$ be a $1$-superminimizer in $\Om$. Then $u^{\wedge}\colon\Om\to (-\infty,\infty]$
	is lower semicontinuous.
\end{theorem}

\begin{lemma}[{\cite[Lemma 3.6]{L-WC}}]\label{lem:solutions from capacity}
	If $x\in X$, $0<r<R<\frac 18 \diam X$, and $A\subset B(x,r)$, then there exists
	$E\subset B(x,R)$ that is a solution of the $\mathcal K_{A,0}(B(x,R))$-obstacle problem,
	with
	\[
	P(E,X)\le \rcapa_1(A,B(x,R)).
	\]
\end{lemma}

We have the following result on the topological components of $F$
(note that this is closely related to Proposition \ref{prop:connected components}).
 
\begin{proposition}[{\cite[Proposition 5.13]{L-Ftype}}]\label{prop:components are subminimizers}
	Suppose $X$ is geodesic.
	Let $F\subset \Om$ be relatively closed
	such that $\ch_F=\ch_F^{\vee}$ is a $1$-subminimizer in $\Om$.
	Denote the components (in the metric topology) of $F$
	with nonzero $\frm$-measure by $F_1,F_2,\ldots$.
	Then $F=\bigcup_{j=1}^{\infty}F_j$
	and each $\ch_{F_k}$ is a $1$-subminimizer in $\Om$.
	Moreover, for every ball $B(y,r)\subset \Om$
	with $0<r<\tfrac14 \diam X$, each component
	$F_k$ intersecting $B(y,r/2)$ satisfies
	\[
	\frac{\frm(F_k\cap B(y,r))}{\frm(B(y,r))}\ge C_1^{-1},
	\]
	where $1\le C_1<\infty$ only depends on $C_d$, $C_P$, and $\lambda$;
	thus there are fewer than $C_1+1$ such components.
\end{proposition}

By Theorem \ref{thm:superminimizers are lsc}, note that for a $1$-subminimizer $\ch_F$ in $\Om$,
the function $\ch_F^{\vee}$ is upper semicontinuous, so that the choice of representative
$\ch_F=\ch_F^{\vee}$ corresponds to choosing the set $F$ to be relatively closed.

Finally, we have the following consequence of the weak Harnack inequality.

\begin{lemma}[{\cite[Lemma 5.9]{L-Ftype}}]\label{lem:smallness in annuli}
	Let $B=B(x,R)$ be a ball with $0<R<\frac{1}{32} \diam X$, and
	suppose that $W\subset B$.
	Let $V\subset 4B$ be a solution of the $\mathcal K_{W,0}(4B)$-obstacle problem
	(whose existence is guaranteed by Lemma \ref{lem:solutions from capacity}).
	Then for all
	$y\in 3 B\setminus 2 B$,
	\[
	\ch_V^{\vee}(y)\le C_2 R \frac{\rcapa_1(W,4B)}{\frm(B)}
	\]
	for some finite constant $C_2=C_2(C_d,C_P,\lambda)\ge 1$.
\end{lemma}

\textbf{Standing assumptions:} We will assume  throughout the paper
that $(X,d,\frm)$ is a complete metric space 
consisting of at least two points, and
is equipped with the doubling Radon measure $\frm$ and supports a
$1$-Poincar\'e inequality.

\section{Constructing a metric measure space}\label{sec:contruct space}

In order to prove our main results on essential connected components and the $1$-fine topology,
it turns out to be useful, given a set $V$ that is $1$-thin at $x$ and a small radius $r>0$, to consider
$B(x,r)\setminus V$ as a metric space in its own right, equipped with a new metric.
Techniques developed in \cite{L-Ftype} prove useful here.

In a metric space $(Z,\widehat{d})$,
we define the \textbf{Mazurkiewicz metric} (connectedness here is of course with respect to the metric $\widehat{d}$-topology)
\begin{equation}\label{eq:widehat d c}
	\widehat{d}_M(x,y):=\inf\{\diam F:\,F\subset Z\textrm{ is a connected set containing }x,y\},
	\quad x,y\in Z,
\end{equation}
where $\diam F:=\sup\{\widehat{d}(z,w)\colon z,w\in F\}$.
Note that if the space $(Z,\widehat{d})$ is geodesic, then
$\widehat{d}=\widehat{d}_M$, while the converse does not generally hold
(consider $\R$ equipped with the metric $\widehat{d}(x,y)=|x-y|^{\beta}$,
$\beta\in (0,1)$).

The following proposition shows the existence of a point on the \emph{strong boundary}.
In \cite[Corollary 3.8]{L-Ftype}, the result was shown with the assumption that $\nu$ is (globally) doubling,
but from the proof it is clear that it is enough to have a slightly weaker
doubling assumption, as below.
Moreover, in \cite{L-Ftype} it was assumed that $Z$ is complete,
but this can be replaced with the assumption that $\overline{B}(y,6r)$ is compact.

\begin{proposition}\label{prop:density points}
	 Consider a metric measure space $(Z,\widehat{d},\nu)$, where $\nu$ is a
	Borel regular outer measure.
	Suppose
    $\widehat{d}_M=\widehat{d}$.
	Let $y\in Z$, $r>0$, and let $E\subset Z$ be a $\nu$-measurable set
	such that
	\[
	0< \nu(E\cap B(y,r))<\nu(B(y,r)).
	\]
	Suppose $\overline{B}(y,6r)$ is compact, and that $\nu$ has the doubling property
    \[
    \nu(B(x,2s))\le \widehat{C}_d\nu(B(x,s))
    \]
    whenever $x\in B(y,6r)$ and $0<s\le r$.
	Then there exists a point $x\in B(y,6 r)$ such that
	\begin{equation}\label{eq:strong boundary point}
		\frac{1}{4 \widehat{C}_d^{12}}\le \liminf_{s\to 0}\frac{\nu(E\cap B(x,s))}{\nu(B(x,s))}
		\le \limsup_{s\to 0}\frac{\nu(E\cap B(x,s))}{\nu(B(x,s))}
		\le 1-\frac{1}{4 \widehat{C}_d^{12}}.
	\end{equation}
\end{proposition}

In our setting, the space $Z$ (as a set) will be a connected set $F\subset X$.
The metric $\widehat{d}$ will itself be the Mazurkiewicz metric (connectedness is in the metric $d$-topology)
\[
d_{M}^F(x,y):=\inf\{\diam K: K\subset F\textrm{ is a connected set containing }x,y\},
\quad x,y\in F.
\]

The following is a variant of \cite[Lemma 6.1]{L-Ftype}.

\begin{lemma}\label{lem:new metric lemma}
	Let $B(x_0,R)\subset X$ be a ball and let $F\subset B(x_0,R)$ be a connected set that is relatively closed in
	$B(x_0,R)$.
	Moreover, suppose that
	for every $x\in F$ and $r>0$ such that $B(x,r)\subset B(x_0,R)$,
	the connected components of $B(x,r)\cap F$
	intersecting $B(x,r/2)$ are finite in number.
	
	Then $d_M^F$ is a metric on $F$ such that $d\le d_M^F$,
	$d_M^F$ induces the same topology on $F$ as $d$, and $(d_M^F)_M=d_M^F$.
\end{lemma}

Note that explicitly, for $x,y\in F$,
\[
(d_{M}^F)_M(x,y)=
\inf\{\diam_{d_M^F} K:\, K\subset F
\textrm{ is a }d_M^F\textrm{-connected set containing }x,y\}.
\]
\begin{proof}
	We check that $d_M^F$ is a metric.
	We obviously have for all $x,y\in F$ that
	\[
	d_M^F(x,y)\le \diam F<\infty.
	\]
	Obviously $d\le d_M^F$ and
	$d_M^F(x,x)=0$ for all $x\in F$.
	If $d_M^F(x,y)=0$ then $d(x,y)=0$ and so $x=y$. Obviously
	also $d_M^F(x,y)=d_M^F(y,x)$ for all $x,y\in F$.
	Finally, take $x,y,z\in F$.
	Take a connected set $K_1\subset F$ containing $x,y$ and a connected set 
	$K_2\subset F$ containing $y,z$.
	Then $K_1\cup K_2\subset F$ is a connected set containing $x,z$
	and so
	\[
	d_M^F(x,z)\le \diam(K_1\cup K_2)\le \diam(K_1)+\diam (K_2).
	\]
	Taking the infimum over $K_1$ and $K_2$, we conclude that the triangle inequality holds.
	Hence $d_M^F$ is a metric on $F$.
	
	To show that the topologies induced on $F$ by $d$ and $d_M^F$ are the same,
	take a sequence $x_j\to x$ with respect to $d$ in $F$.
	Fix an arbitrary $\eps>0$ sufficiently small that $B(x,\eps/2)\subset B(x_0,R)$. Consider the
	components of $B(x,\eps/2)\cap F$ intersecting $B(x,\eps/4)$.
	By assumption there are only finitely many.
	Each of them not containing $x$ is at a nonzero distance from $x$ (since $F$ is relatively closed) and
	so for large $j$, every $x_j$ belongs to the component containing $x$; denote it by $F_1$.
	For such $j$, we have
	\[
	d_M^F(x_j,x)\le \diam F_1\le \eps. 
	\]
	We conclude that $x_j\to x$ also with respect to $d_M^F$.
	Since we had $d\le d_M^F$, it follows that the topologies are the same.
	
	If $x,y\in F$, and $\eps'>0$, we can take a connected set  $K$
	containing $x$ and $y$, with $\diam K<d_M^F(x,y)+\eps'$.
	Since the topologies are the same, the set $K$ is still a connected set
    in the metric space $(F,d_M^F)$, and for every $z,w\in K$,
	\[
	d_M^F(z,w)\le \diam K< d_M^F(x,y)+\eps'.
	\]
	It follows  that $\diam_{d_M^F} K\le d_M^F(x,y)+\eps'$,
	and so $(d_M^F)_M(x,y)\le d_M^F(x,y)+\eps'$, showing that $(d_M^F)_M=d_M^F$.
\end{proof}

We will rely on the following \emph{Cartan property}, which is of independent interest.
Since its proof is quite long, we will postpone it as well as some discussion on the property to
Section \ref{sec:Cartan}.

Recall that by our standing assumptions, we have $\frm(X)>0$ (possibly $\infty$),
and $\frm(\{x\})=0$ for every $x\in X$ (see e.g.\ \cite[Corollary 3.9]{BB}).

\begin{theorem}\label{thm:strong Cartan property}
	Let $W\subset X$ be open and let $x\in X\setminus W$ such that $W$
	is $1$-thin at $x$.
	Choose $0<R<\frac{1}{32} \diam X$ sufficiently small that
	\begin{equation}\label{eq:space measure}
	\frac{\frm(B(x,2R))}{\frm(X)}<C_0^{-1}
	\end{equation}
and
	\begin{equation}\label{eq:small capacity}
	\sup_{0<r\le R}\frac{\rcapa_1(W\cap B(x,r),B(x,2r))}{\rcapa_1(B(x,r),B(x,2r))}<C_0^{-1},
	\end{equation}
	where
	\[
	C_0:=(20C_{SP} C_2 C_S C_d^{1+2\lceil\log_2(2\lambda)\rceil})^{2Q}.
	\]
	Denote $B:=B(x,R)$.
	Let $\ch_V=\ch^\wedge_{V}$ be a solution to the
	$\mathcal K_{W\cap B,0}(4B)$-obstacle problem (as guaranteed by Lemma \ref{lem:solutions from capacity}).
	Then $V$ is open, contained in $2B$,
	$\ch_V^{\wedge}=1$ in $W\cap B$,
	$\ch_V^{\vee}(x)=0$, and
	\begin{equation}\label{eq:V cap density zero}
	\lim_{r\to 0}\frac{\rcapa_1(V\cap B(x,r),B(x,2r))}{\rcapa_1(B(x,r),B(x,2r))}=0.
	\end{equation}
\end{theorem}

In essence, the theorem says that if the obstacle $W$ is $1$-thin at $x$, then so is the superminimizer $V$. 

Now we get the following result on constructing a metric measure space;
this is a variant of \cite[Proposition 6.6]{L-Ftype}.
We denote by $B_F(y,r)$ an open ball in $F$, defined with respect to the metric $d_M^F$.

When $X$ is geodesic, in the $1$-Poincar\'e inequality we can choose
$\lambda=1$, see e.g.\ \cite[Theorem 4.39]{BB}. This may change the constant $C_P$;
whenever we assume that $X$ is geodesic, by $C_P$ we mean the constant corresponding to the choice
$\lambda=1$.

\begin{proposition}\label{prop:constructing the quasiconvex space}
	Suppose $X$ is geodesic.
	Let $x_0\in X$, let $W\subset X$ be an open set that is $1$-thin at $x_0$,
    and choose $0<R<\frac{1}{32} \diam X$ sufficiently small that
    \[
	\frac{\frm(B(x_0,2R))}{\frm(X)}<C_0^{-1}
	\]
    and
	\[
	\sup_{0<r\le R}\frac{\rcapa_1(W\cap B(x_0,r),B(x_0,2r))}{\rcapa_1(B(x_0,r),B(x_0,2r))}
    <\frac{1}{C_0 C_1 C_d C_P}.
	\]
Denote $B:=B(x_0,R)$.
	Then we find an open set $V$ with
	$W\cap B\subset V\subset 2B$ and
the following hold:
	denote by $F$ the component of $B(x_0,R)\setminus V$ containing $x_0$;
	then $x_0\in \fint F$;
	\begin{equation}\label{eq:capacity of F}
		\rcapa_1(\tfrac 12 B\setminus F,B)\le 9C' C_r \rcapa_1(W\cap B,2B);
	\end{equation}
	the space $(F,d_M^F,\frm)$ is a locally complete metric space with
	$(d_M^F)_M=d_M^F$ and with the topology inherited from $(X,d,\mu)$,   
    $\frm$ in $F$ is a Borel regular outer measure
	with the doubling property
    \[
    \frm(B_F(y,2r))\le C_1 C_d^2\frm(B_F(y,r))
    \]
    whenever $B(y,r)\subset B(x_0,R)$ with $y\in F$,
	and for such balls $B(y,r)$ also
	\begin{equation}\label{eq:BF lower bound}
	\frac{\frm(B_F(y,r))}{\frm(B(y,r))}\ge (C_1 C_d)^{-1}.
	\end{equation}
\end{proposition}
	
\begin{proof}
	Take a solution $\ch_V=\ch^{\wedge}_V$ of the
	$\mathcal K_{W\cap B,0}(4B)$-obstacle problem.
	By Theorem \ref{thm:strong Cartan property} we have that $V\subset 2B$ is open, and
	$V$ is $1$-thin at $x_0$, so that by \eqref{eq:fine int}, also
	\begin{equation}\label{eq:x in fine interior}
	x_0\in \fint(X\setminus V).
	\end{equation}
	By Lemma \ref{lem:solutions from capacity}, we also know that
	\begin{equation}\label{eq:V and W}
	P(V,X)\le \rcapa_1(W\cap B,4B)\le \rcapa_1(W\cap B,2B)
    <\frac{\rcapa_1(B,2B)}{C_0 C_1 C_d C_P}.
	\end{equation}
Moreover, by \eqref{eq:variational one and BV capacity} we get
\begin{equation}\label{eq:capacity of V}
	\rcapa_1(V,4B)
	\le C_r \rcapa_{\BV}^{\vee}(V,4B)
	\le C_r P(V,X).
\end{equation}
	Denote by $F$ the connected component of $B\setminus V$ containing $x_0$;
	$F$ is a relatively closed subset of $B$.
    By Proposition \ref{prop:solutions are superminimizers} and Proposition \ref{prop:components are subminimizers}
    (with $\Om=B(x_0,R)$)
	we know that $\ch_F=\ch_F^{\vee}$ is a $1$-subminimizer in $B(x_0,R)$, and
	$B(x_0,R)\setminus V$ is the union of its connected components $\bigcup_{j}F_j$, where we can choose $F=F_1$,
	such that each component
	$F_j$ intersecting $B(x_0,R/2)$ satisfies
	\[
		\frac{\frm(F_j)}{\frm(B(x_0,R))}\ge C_1^{-1},
	\]
	and thus there are fewer than $C_1+1$ such components.
	Suppose there is at least one other such component apart from $F_1$.
	By Proposition \ref{prop:connected components} and the relative isoperimetric inequality, we know that
    (recall that in the Poincar\'e inequality we now have $\lambda=1$)
	\begin{align*}
	P(X\setminus V,B(x_0,R))
	&=\sum_{j}P(F_j,B(x_0,R))\\
	&\ge C_1^{-1}C_P^{-1}\frac{\frm(B(x_0,R))}{R},
	\end{align*}
	contradicting \eqref{eq:V and W} and \eqref{eq:cap estimate}.
	Hence
	\begin{equation}\label{eq:F and V}
	B(x_0,R/2)\cap F=B(x_0,R/2)\setminus V.
	\end{equation}
	Then from \eqref{eq:x in fine interior} we get the fact that $x_0\in \fint F$.
	Moreover,
	\begin{align*}
	\rcapa_1(\tfrac 12 B\setminus F,B)
	&\le 9C'\rcapa_1(\tfrac 12 B\setminus F,4B)\quad \textrm{by }\eqref{eq:capacities}\\
	&= 9C'\rcapa_1(\tfrac 12 B\cap V,4B)\quad \textrm{by }\eqref{eq:F and V}\\
	&\le 9C' C_r \rcapa_1(W\cap B,2B)
	\quad \textrm{by }\eqref{eq:capacity of V},\eqref{eq:V and W},
	\end{align*}
	which is \eqref{eq:capacity of F}.

Suppose $y\in F$ and $0<r\le R$ is small enough that $B(y,r)\subset B(x_0,R)$.
	We know that $\ch_F$ is a $1$-subminimizer in $B(x_0,R)$, and then 
	also in $B(y,r)\subset B(x_0,R)$.
	By another application of Proposition \ref{prop:components are subminimizers}  we know that
	$F\cap B(y,r)$ is the union of its connected components $\bigcup_{j}F_j'$,
	such that each component
	$F_j'$ intersecting $B(y,r/2)$ satisfies
	\begin{equation}\label{eq:subminimizer component measure lower bound}
		\frac{\frm(F_j')}{\frm(B(y,r))}\ge C_1^{-1}
	\end{equation}
	and thus there are fewer than $C_1+1$ such components.
	Now Lemma \ref{lem:new metric lemma} gives that $(F,d_M^F,\frm)$
	is a locally complete metric space, $d\le d_M^F$, the topologies induced
	by $d$ and $d_M^F$ are the same and so the topology is inherited from $(X,d,\mu)$, 
    and $(d_M^F)_M=d_M^F$.
	Note that $\frm$ restricted to the set $F$ is still a Borel
	regular outer measure, see \cite[Lemma 3.3.11]{HKST15}.
	Since the topologies induced by $d$ and $d_M^F$
	are the same, $\frm$ remains a Borel regular outer measure in the metric measure space $(F,d_M^F,\frm)$.
	
	For any ball
	$B(y,r)\subset B(x_0,R)$ with $y\in F$,
	by \eqref{eq:subminimizer component measure lower bound} we have that
	\[
	\frac{\frm(B_F(y,2r))}{\frm(B(y,r))}\ge C_1^{-1},
	\]
	and so in fact
	\[
	\frac{\frm(B_F(y,2r))}{\frm(B(y,2r))}\ge (C_1 C_d)^{-1}.
	\]
	We conclude that whenever $B(y,r)\subset B(x_0,R)$, we have
	\[
	\frac{\frm(B_F(y,r))}{\frm(B(y,r))}\ge (C_1 C_d)^{-1},
	\]
	which is \eqref{eq:BF lower bound}.
    We also have
	\[
	\frac{\frm(B_F(y,2r))}{\frm(B_F(y,r))}\le C_1 C_d\frac{\frm(B(y,2r))}{\frm(B(y,r))}
	\le C_1 C_d^2,
	\]
	giving the desired doubling property.
\end{proof}

We have also the following variant of
Proposition \ref{prop:constructing the quasiconvex space}.

\begin{proposition}\label{prop:constructing the quasiconvex space alternative}
	Suppose $X$ is geodesic.
	Let $x_0\in X$, let $W\subset X$ be an open set,
	let $0<R<\frac{1}{32} \diam X$, and denote $B=B(x_0,R)$. Suppose
	\[
	\frac{\rcapa_1(W\cap B,2B)}{\rcapa_1(B,2B)}
	<\frac{1}{C_0 C_1 C_d C_P}.
	\]
	Then we find an open set $V$ with
	$W\cap B\subset V\subset 2B$ and
the following hold:
denote by $F$ the component of $B(x_0,R)\setminus V$ that has the largest $\frm$-measure
among components intersecting $B(x_0,R/2)$; then
\begin{equation}\label{eq:capacity of F 2}
\rcapa_1(\tfrac 12 B\setminus F,B)\le 9C' C_r \rcapa_1(W\cap B,2B);
\end{equation}
the space $(F,d_M^F,\frm)$ is a locally complete metric space with
$(d_M^F)_M=d_M^F$, $\frm$ in $F$ is a Borel regular outer measure
with the doubling property
\[
\frm(B_F(y,2r))\le C_1 C_d^2\frm(B_F(y,r))
\]
whenever $B(y,r)\subset B(x_0,R)$ with $y\in F$,
and for such balls $B(y,r)$ also
\[
\frac{\frm(B_F(y,r))}{\frm(B(y,r))}\ge (C_1 C_d)^{-1}.
\]
\end{proposition}

\begin{proof}
The proof is mutatis mutandis the same as that of Proposition \ref{prop:constructing the quasiconvex space},
except with the sentences referring to $\fint (X\setminus V)$ or $\fint F$ removed.
Note that we still have $V\subset 2B$, which follows exactly as in the first paragraph
of the proof of Theorem \ref{thm:strong Cartan property}.
\end{proof}

\section{The \texorpdfstring{$1$}{1}-fine topology is locally connected}\label{sec:locally connected}

In this section we show that the $1$-fine topology is locally connected.
In the previous sections, we have developed some of our theory under the assumption that $X$ is geodesic.
However, it is a standard fact that this assumption is not very restrictive.
Indeed, as \label{quasiconvex and geodesic}a complete metric space equipped with a doubling measure and supporting a Poincar\'e
inequality, $X$ is \emph{quasiconvex}, meaning that for every
pair of points $x,y\in X$ there is a curve $\gamma\colon [0,\ell_{\gamma}]\to X$
with $\gamma(0)=x$,
$\gamma(\ell_{\gamma})=y$, and $\ell_{\gamma}\le Cd(x,y)$,
where $C$
only depends on $C_d$ and $C_P$, see e.g.\ \cite[Theorem 4.32]{BB}.
Thus a bi-Lipschitz change in the metric gives a geodesic space that is still complete,
the measure is doubling, and the space supports a $1$-Poincar\'e inequality
(see \cite[Section 4.7]{BB}).

\begin{lemma}\label{lem:invariance}
The $1$-fine topology is invariant under a bi-Lipschitz change of the metric.
\end{lemma}
\begin{proof}
Consider a different metric $d'$ for which we have $L^{-1}d(x,y)\le d'(x,y)\le Ld(x,y)$
for all $x,y\in X$ and some finite $L\ge 1$.
Denote the capacity with respect to $d'$ by $\rcapa_1'$.
Note that if $g$ is an upper gradient of a function $u$ in the space $(X,d',\frm)$,
then $Lg$ is an upper gradient of $u$ in $(X,d,\frm)$.
Recalling \eqref{eq:capacities}, we have
\begin{align*}
\rcapa_1(A\cap B(x,r),B(x,2r))
&\le C'(1+2L^2)\rcapa_1(A\cap B(x,r),B(x,2L^2 r))\\
&\le C'(1+2L^2)\rcapa_1(A\cap B'(x,Lr),B'(x,2L r))\\
&\le LC'(1+2L^2)\rcapa_1'(A\cap B'(x,Lr),B'(x,2L r)).
\end{align*}
The opposite inequality is analogous.
This shows that a set $A\subset X$ is $1$-thin in the space $(X,d,\frm)$ if and only if it is $1$-thin
in the space $(X,d',\frm)$.
\end{proof}

\begin{lemma}\label{lem:components invariance}
The two-sidedness property is invariant under
a bi-Lipschitz change of the metric,
and if $X$ is isotropic with respect to both metrics,
then the essential connected components of a set of finite perimeter are also invariant. 
\end{lemma}
\begin{proof}
Finite perimeter and the measure-theoretic boundary are easily seen to be invariant under a bi-Lipschitz change
of the metric, and the Hausdorff measure $\mathcal H$ remains comparable under such a change.
Thus the two-sidedness property is clearly invariant.

If $X$ is isotropic with respect to both metrics,
then by Lemma \ref{lem:measure theoretic boundaries}
we have that the essential connected components
of a set of finite perimeter are also invariant.
\end{proof}

\begin{proposition}\label{prop:open basis}
	The $1$-fine topology is generated by a basis of $1$-finely connected sets.
\end{proposition}
\begin{proof}
    By Lemma \ref{lem:invariance}, we can assume that $X$ is geodesic.
    With this assumption, here and in later proofs we understand various constants such as $C_d$ and $C_P$
    to be those in the geodesic space, perhaps different from the original space, and we can assume that $\lambda=1$.
    
	Consider a $1$-finely open set $U\subset X$. Fix $x_0\in U$.
	By Lemma \ref{lem:open thin set}, we find an open set $W\supset X\setminus U$ that is $1$-thin at $x_0$.
	Choose $0<R<\frac{1}{32} \diam X$ sufficiently small that
    \[
	\frac{\frm(B(x_0,2R))}{\frm(X)}<C_0^{-1}
	\]
    and
	\[
	\sup_{0<r\le R}\frac{\rcapa_1(W\cap B(x_0,r),B(x_0,2r))}{\rcapa_1(B(x_0,r),B(x_0,2r))}
    <\frac{1}{C_0 C_1C_dC_P}.
	\]
	Denoting $B=B(x_0,R)$, we can choose a connected set $F\subset B\setminus W$
	as in Proposition \ref{prop:constructing the quasiconvex space}.
	We know that $x_0\in \fint F$.
	We will show that $F$ is $1$-finely connected.
	
Suppose instead that $F$ is not $1$-finely connected.
Then we have
$F\subset V_1\cup V_2$,
where $V_1,V_2\subset B(x_0,R)$ are $1$-finely open sets with $V_1\cap V_2=\emptyset$, and
$V_1\cap F$ and  $V_2\cap F$ are both nonempty.
By connectedness of $F$ (in the metric $d$-topology and hence also
in the $d_M^F$-topology), there is a boundary point $y\in \partial_F V_1$,
i.e.\ a point in the boundary of the set $V_1$ in the metric space
$(F,d^F_M)$.
Let $s>0$ be so small that $B(y,7s)\subset B(x_0,R)$.
There are points $z_1\in B_{F}(y,s)\cap V_1$ and $z_2\in B_{F}(y,s)\cap V_2$.
By \eqref{eq:thinness and measure thinness} and \eqref{eq:BF lower bound}, we have
\[
\lim_{r\to 0}\frac{\frm(V_1\cap B_F(z_1,r))}{\frm(B_F(z_1,r))}=1
\quad\textrm{and}\quad
\lim_{r\to 0}\frac{\frm(V_2\cap B_F(z_2,r))}{\frm(B_F(z_2,r))}=1.
\]
Thus
	\[
	0<\frac{\frm(V_1\cap B_{F}(y,s))}{\frm(B_{F}(y,s))}<1.
	\]
	By Proposition \ref{prop:constructing the quasiconvex space}
	we know that
	$\frm$ has the doubling property
    \[
    \frm(B_{F}(z,2r))\le C_1 C_d^2\frm(B_{F}(z,r))
    \]
    for every $z\in B(y,6s)$ and $0<r\le s$, and then also
    for every $z\in B_F(y,6s)$ and $0<r\le s$.
	Moreover, in the space $(F,d_M^F)$, $\overline{B}_F (y,6s)$ is a closed subset of
	the set $\overline{B}(y,6s)\cap F$ that is compact in $(X,d)$ and then also in $(F,d_M^F)$ since the 
	topologies are the same. Thus $\overline{B}_F (y,6s)$ is compact.
	Then by Proposition \ref{prop:density points}, we find a point $z\in B_F(y,6s)$
	(in particular, $z\in F$)
	such that
	\[
		\frac{1}{4 (C_1 C_d^2)^{12}}\le \liminf_{r\to 0}\frac{\frm(V_1 \cap  B_{F}(z,r))}{\frm(B_{F}(z,r))}
		\le \limsup_{r\to 0}\frac{\frm(V_1 \cap B_{F}(z,r))}{\frm(B_{F}(z,r))}
		\le 1-\frac{1}{4 (C_1 C_d^2)^{12}}.
	\]
	By \eqref{eq:BF lower bound}, we obtain
		\[
	\frac{1}{4 (C_1 C_d^2)^{13}}\le \liminf_{r\to 0}\frac{\frm(V_1 \cap B(z,r))}{\frm(B(z,r))}
	\le \limsup_{r\to 0}\frac{\frm(V_1 \cap B(z,r))}{\frm(B(z,r))}
	\le 1-\frac{1}{4  (C_1 C_d^2)^{13}}.
	\]
	The same inequalities hold for $V_2$.
	By \eqref{eq:thinness and measure thinness} we have
	that $z\notin V_1\cup V_2$, a contradiction. Thus $F$ is $1$-finely connected.
	
	We conclude that for every $x\in U$ we find a $1$-finely connected set $F_x\subset U$
	such that $x\in \fint F_x$.
	Let $V$ be the $1$-finely connected component of $U$ containing $x$. Now necessarily
	$V$ is $1$-finely open.
\end{proof}

\begin{proposition}\label{prop:finite perimeter}
	Let $E\subset X$ be a set of finite perimeter.
	Connected components of $\fint I_E$
	in the $1$-fine topology, as well as any union of them, are sets of finite perimeter.
\end{proposition}

\begin{proof}
    A set having finite perimeter is clearly invariant under a bi-Lipschitz change of the metric,
    and so is the measure-theoretic interior $I_E$.
    By Lemma \ref{lem:invariance}, the $1$-fine topology is invariant as well.
    By the discussion at the beginning of Section \ref{sec:locally connected}, we can assume that $X$ is geodesic.
    
	By Proposition \ref{prop:open basis}, we know that the connected
	components of $\fint I_E$ in the $1$-fine topology are also $1$-finely open sets.
	Each one has nonzero measure due to \eqref{eq:thinness and measure thinness},
	and so there are at most countably many of them,
	and thus we have the decomposition into connected components
	\begin{equation}\label{eq:countable union}
		\fint I_E=\bigcup_{j}V_j,
	\end{equation}
	where the union is either finite or countable.
	Let $V'=V_k$ for some $k$. Suppose $x\in \partial^1 V'$.
	By the characterization \eqref{eq:fine bdry char}, we know that
	$x\in b_1 V'\subset b_1 I_E\subset \overline{I_E}^1$.
	If the point $x$ were in some other connected component $V''$, it would follow that
	\[
	\limsup_{r\to 0}\frac{\rcapa_1(B(x,r)\cap V',B(x,2r))}{\rcapa_1(B(x,r),B(x,2r))}
	\le \limsup_{r\to 0}\frac{\rcapa_1(B(x,r)\setminus  V'',B(x,2r))}{\rcapa_1(B(x,r),B(x,2r))}=0,
	\]
	contradicting the fact that $x\in b_1 V'$.
	In total, $x\in X\setminus \bigcup_j V_j$, that is,
	$x\notin \fint I_E$, and so $x\in \overline{I_E}\setminus \fint I_E= \partial^1 I_E$.
	We conclude that $\partial^1 V'\subset \partial^1 I_E$.
	By Theorem \ref{thm:Federer fine} we know that $\mathcal H(\partial^1 I_E)<\infty$,
	and so also $\mathcal H(\partial^1 V')<\infty$.
	We have $V'\subset I_{V'}$ by \eqref{eq:thinness and measure thinness}.
    On the other hand, if $y\in I_{V'}$, again by \eqref{eq:thinness and measure thinness}
    we know that $y$ cannot belong to any other set $V_j$, and so
    \begin{align*}
    \mathcal H(I_{V'}\setminus V')
    &= \mathcal H(I_{V'}\setminus \fint I_E)\\
    &\le \mathcal H(I_{\fint I_E}\setminus \fint I_E)\quad \textrm{by }\eqref{eq:countable union}\\
    &= \mathcal H(I_{E}\setminus \fint I_E)\quad \textrm{by Proposition }\ref{prop:IE and OE}\\
    &=0
    \end{align*}
	by another application of Proposition \ref{prop:IE and OE}.
	From Lemma \ref{lem:cap negligible} we get
	\[
	\mathcal H(\partial^1 V' \Delta \partial^1 I_{V'})=0.
	\]
	Thus also $\mathcal H(\partial^1 I_{V'})<\infty$.
	Now by Theorem \ref{thm:Federer fine},
	it follows that $V'$ is a set of finite perimeter.
	
	From \eqref{eq:fine boundary and measure}
	and \eqref{eq:strong boundary and measure}, for each $j$ 
	we get $\mathcal H(\partial^1 I_{V_j}\setminus \Sigma_{\gamma}V_j)=0$.
	Each point $x$ can only belong to fewer than $1/\gamma+1$ sets $\Sigma_{\gamma}V_j$.
	It follows that
	\[
	\sum_{j}\mathcal H(\partial^1 I_{V_j})
	= \sum_{j}\mathcal H(\Sigma_{\gamma}V_j)
	\le ( 1/\gamma+1)\mathcal H(\partial^1 I_E).
	\]
	From \eqref{eq:def of theta} it follows that
	\begin{equation}\label{eq:perimeter countable sum}
	\sum_{j} P(V_j,X)\le
    C_d( 1/\gamma+1)\mathcal H(\partial^1 I_E)<\infty.
	\end{equation}
	Thus for any union $\bigcup_{j\in I}V_j$, using the subadditivity
    \eqref{eq:perimeter of union} we get
	\[
	P\Big(\bigcup_{j\in I}V_j,X\Big)
	\le \sum_{j\in I} P(V_j,X)\le C_d( 1/\gamma+1)\mathcal H(\partial^1 I_E)<\infty.
	\]
This completes the proof.
\end{proof}

Now we obtain the following improvement of Proposition \ref{prop:open basis}.
\begin{theorem}\label{prop:open basis improved}
	The $1$-fine topology is generated by a basis of $1$-finely connected sets of finite perimeter.
\end{theorem}
\begin{proof}
Take a $1$-finely open set $V'\subset X$. Fix $x_0\in V'$.
By Lemma \ref{lem:open thin set},
we find an open set $W\supset X\setminus V'$ that is $1$-thin at $x_0$.
Choose $0<R<\frac{1}{32} \diam X$ sufficiently small that
	\[
    \frac{\frm(B(x_0,2R))}{\frm(X)}<C_0^{-1}
	\]
and
	\[
	\sup_{0<r\le R}\frac{\rcapa_1(W\cap B(x_0,r),B(x_0,2r))}{\rcapa_1(B(x_0,r),B(x_0,2r))}<C_0^{-1}.
	\]
    By the coarea formula \eqref{eq:coarea} we know that $P(B(x_0,s),X)<\infty$
    for a.e.\ $s>0$, and so we can assume also that $P(B(x_0,R/2),X)<\infty$.
    As usual, denote $B=B(x_0,R)$.
	Take a solution $\ch_V=\ch^{\wedge}_V$ of the
	$\mathcal K_{W\cap B,0}(4B)$-obstacle problem;
	by Theorem \ref{thm:strong Cartan property} we have that $V\subset 2B$ is open, and
	$V$ is $1$-thin at $x_0$, so that by \eqref{eq:fine int},
	$x_0\in \fint(X\setminus V)$.
	By Lemma \ref{lem:solutions from capacity}, we also know that
	\[
	P(V,X)\le \rcapa_1(W\cap B,4B)<\infty.
	\]
	Then by \eqref{eq:perimeter subadditivity},
    $E:=\tfrac 12 B\setminus V$ is a set of finite perimeter with $x_0\in\fint E$.
    By \eqref{eq:thinness and measure thinness} we know that $\fint E\subset I_E$,
    and so $\fint E\subset \fint I_E$. Thus also $x_0\in \fint I_E$.
    By Proposition \ref{prop:finite perimeter} we know that the $1$-finely connected component
    of $\fint I_E$ containing $x_0$ is a set of finite perimeter.
    Finally,
    \[
    \fint I_E\subset \overline{E}\subset \overline{\tfrac 12 B}\setminus V \subset V'.
    \]
    Thus $\fint I_E$ is the required basis set.
\end{proof}

We also note the following result.

\begin{proposition}\label{prop:capacity sum}
Let $E\subset X$ be a set of finite perimeter with $\frm(E)<\infty$,
and write $\fint I_E$ as the union of its $1$-finely connected components
\[
\fint I_E=\bigcup_{j}V_j.
\]
Then
\[
\sum_{j}\capa_1(V_j)<\infty.
\]
\end{proposition}
\begin{proof}
By \eqref{eq:finely and quasiopen} and Proposition \ref{prop:IE and OE},
each set $V_j$ is $1$-quasiopen, and so we find open sets $G_j$ such that each $V_j\cup G_j$
is open and $\capa_1(G_j)<2^{-j}$.
For each $j\in\N$, we find a function $u_j\in N^{1,1}(X)$ such that
$0\le u_j\le 1$, $u_j\ge 1$ in $G_j$, and
$\Vert u_j\Vert_{N^{1,1}(X)}< 2^{-j}$. Then also $u_j\in \BV(X)$ and
\[
\Vert u_j\Vert_{L^1(X)}+\Vert Du_j\Vert(X)< 2^{-j},
\]
and by the coarea formula \eqref{eq:coarea} and Cavalieri's principle, we find $t\in (0,1)$ such that
\[
\frm(\{u_j>t\})+P(\{u_j>t\},X)< 2^{-j}.
\]
Note that $V_j\cup \{u_j>t\}$ contains $V_j\cup G_j$ which is a neighborhood of $V_j$.
Then
\begin{align*}
\capa_{\BV}(V_j\cup G_j)
&\le \frm(V_j\cup \{u_j>t\})+P(V_j\cup \{u_j>t\},X)\\
&\le \frm(V_j)+\frm(\{u_j>t\})+P(V_j,X)+P(\{u_j>t\},X)\\
&\le \frm(V_j)+P(V_j,X)+2^{-j},
\end{align*}
and so by \eqref{eq:cap and BV cap},
\begin{align*}
\sum_{j}\capa_1(V_j)
&\le C_r'\sum_{j}\capa_{\BV}(V_j\cup G_j)\\
&\le C_r'\sum_{j}(\frm(V_j)+P(V_j,X)+2^{-j})\\
&\le C_r'\big(\frm(E)+1+\sum_j P(V_j,X)\big)\\
&<\infty
\end{align*}
by \eqref{eq:perimeter countable sum}.
\end{proof}

\section{The finely connected components are the essential connected components}\label{sec:main result on components}

In this section we prove that the $1$-finely connected components
coincide with the essential connected components of a set of finite perimeter.

We start with the following lemma.

\begin{lemma}\label{lem:continuity argument}
Suppose $0<\delta<1/(15 C_d^2 C' C'')$,
$x\in X$, and $A_1,A_2\subset X$ are disjoint sets with 
\[
\lim_{r\to 0}\frac{\rcapa_1(B(x,r)\setminus (A_1\cup A_2),B(x,2r))}{\rcapa_1(B(x,r),B(x,2r))}=0,
\]
and
\[
\limsup_{r\to 0}\frac{\rcapa_1(A_1\cap B(x,r),B(x,2r))}{\rcapa_1(B(x,r),B(x,2r))}>5 C_d^2 C' C''\delta
\]
and 
\[
\limsup_{r\to 0}\frac{\rcapa_1(A_2\cap B(x,r),B(x,2r))}{\rcapa_1(B(x,r),B(x,2r))}>5 C_d^2 C' C''\delta.
\]
Then
\[
\limsup_{r\to 0}\frac{\min\{\rcapa_1(A_1\cap B(x,r),B(x,2r)),\rcapa_1(A_2\cap B(x,r),B(x,2r))\}}{\rcapa_1(B(x,r),B(x,2r))}>\delta.
\]
\end{lemma}
\begin{proof}
For all sufficiently small $r>0$, we have
\[
\frac{\rcapa_1(A_1\cap B(x,r),B(x,2r))}{\rcapa_1(B(x,r),B(x,2r))}
+\frac{\rcapa_1(A_2\cap B(x,r),B(x,2r))}{\rcapa_1(B(x,r),B(x,2r))}>\frac{2}{3}.
\]
Thus we can find an arbitrarily small $R>0$ and $r\in [R/2,R]$ such that
\[
\frac{\rcapa_1(A_1\cap B(x,R),B(x,2R))}{\rcapa_1(B(x,R),B(x,2R))}>5 C_d^2 C' C''\delta
\]
and
\[
\frac{\rcapa_1(A_2\cap B(x,r),B(x,2r))}{\rcapa_1(B(x,r),B(x,2r))}>5 C_d^2 C' C''\delta.
\]
Then by \eqref{eq:cap estimate},
\begin{align*}
&\frac{\rcapa_1(A_2\cap B(x,R),B(x,2R))}{\rcapa_1(B(x,R),B(x,2R))}\\
&\quad \ge (C_d^2 C'')^{-1}\frac{\rcapa_1(A_2\cap B(x,r),B(x,2R))}{\rcapa_1(B(x,r),B(x,2r))}\\
&\quad \ge (5 C_d^2 C' C'')^{-1}\frac{\rcapa_1(A_2\cap B(x,r),B(x,2r))}{\rcapa_1(B(x,r),B(x,2r))}
\quad\textrm{by }\eqref{eq:capacities}\\
&\quad >\delta.
\end{align*}
Since $R>0$ can be chosen arbitrarily small, we have the conclusion.
\end{proof}

Recall the notation $\Theta(E,x)$ for measure densities from \eqref{eq:upper measure density}.

\begin{theorem}\label{thm:iff}
Suppose $X$ has the two-sidedness property. Let $E\subset X$ be such that $P(E,X)<\infty$.
Then the set $E$ is indecomposable
if and only if $\fint I_E$ is 1-finely connected.
\end{theorem}

\begin{proof}
    By Lemma \ref{lem:invariance} and Lemma \ref{lem:components invariance},
    and recalling that the two-sidedness property implies that the space is isotropic
    (see \cite[Lemma 6.2]{AFP}),
    we can assume that $X$ is geodesic.
	By Proposition \ref{prop:IE and OE},
	there exists a $\mathcal H$-negligible set $N\subset X$ such that
	\begin{equation}\label{eq:definition of N}
	\fint I_E=I_E\setminus N.
	\end{equation}
	
	\textbf{Part 1.}
    Suppose $V':=\fint I_E$ is $1$-finely connected. We prove that $E$ is indecomposable.
    
	Suppose instead that we had $V'=G\cup H$ for some sets $G,H$ with $G\cap H=\emptyset$,
	$\frm(G)>0$, $\frm(H)>0$, and $P(V',X)=P(G,X)+P(H,X)$.
	Without altering any of the properties given in the previous sentence, we can modify $G$ and $H$ in a set of
	$\frm$-measure zero in such a way that we remove from $H$ all points $x$ for which $\Theta^*(H,x)<1/2$
	and add them to $G$, and vice versa.
	By \eqref{eq:thinness and measure thinness} we have that
	$\Theta(V,x)=1$ for all $x\in V$, and so we end up with sets $G$ and $H$
	such that $\Theta^*(G,x)\ge 1/2$ for all $x\in G$ and $\Theta^*(H,x)\ge 1/2$ for all $x\in H$.
	
	Take $x_0\in G$. Since $V'=G\cup H$, we know that
	\[
	\lim_{r\to 0}\frac{\rcapa_1(B(x_0,r)\setminus (G\cup H),B(x_0,2r))}{\rcapa_1(B(x_0,r),B(x_0,2r))}=0.
	\]
    From the fact that $\Theta^*(G,x_0)\ge 1/2>0$, by
    \eqref{eq:thinness and measure thinness} we get
    \begin{equation}\label{eq:capacity of G}
	\limsup_{r\to 0}\frac{\rcapa_1(G\cap B(x_0,r),B(x_0,2r))}{\rcapa_1(B(x_0,r),B(x_0,2r))}>0.
	\end{equation}
	Suppose we had
	\begin{equation}\label{eq:capacity of H}
	\limsup_{r\to 0}\frac{\rcapa_1(H\cap B(x_0,r),B(x_0,2r))}{\rcapa_1(B(x_0,r),B(x_0,2r))}>0.
	\end{equation}
    From Lemma \ref{lem:continuity argument}, we get for some $\delta>0$
\begin{equation}\label{eq:capacity of G and H}
\limsup_{r\to 0}\frac{\min\{\rcapa_1(G\cap B(x,r),B(x,2r)),\rcapa_1(H\cap B(x,r),B(x,2r))\}}{\rcapa_1(B(x,r),B(x,2r))}>\delta.
\end{equation}
	We claim that 
	\begin{equation}\label{eq:H and G}
	\limsup_{r\to 0}\frac{\rcapa_1((\partial^*H\cap \partial^*G)\cap B(x_0,r),B(x_0,2r))}{\rcapa_1(B(x_0,r),B(x_0,2r))}>0.
	\end{equation}
	Suppose instead that
	\[
	\lim_{r\to 0}\frac{\rcapa_1((\partial^*H\cap \partial^*G)\cap B(x_0,r),B(x_0,2r))}{\rcapa_1(B(x_0,r),B(x_0,2r))}=0.
	\]
    By \eqref{eq:capa x positive} we know that either
    $x_0\notin \partial^*H\cap \partial^*G$ or $\capa_1(\{x_0\})=0$ (or both).
    \bigskip

\textbf{Case a: $x_0\notin \partial^*H\cap \partial^*G$.}
	By Lemma \ref{lem:open thin set}, we find an open set $W$ containing $\partial^*H\cap \partial^*G$
	and $X\setminus (G\cup H)$, such that $W$ is also $1$-thin at $x_0$.
	Choose $0<R<\frac{1}{32} \diam X$ sufficiently small that
    \[
	\frac{\frm(B(x_0,2R))}{\frm(X)}<C_0^{-1}
	\]
    and
	\[
	\sup_{0<r\le R}\frac{\rcapa_1(W\cap B(x_0,r),B(x_0,2r))}{\rcapa_1(B(x_0,r),B(x_0,2r))}
    <\min\left\{\frac{\delta}{9 C_d^2 C_r C' C''},\frac{1}{C_0 C_1C_dC_P}\right\},
	\]
	while with $B=B(x_0,R)$, by \eqref{eq:capacity of G and H} we can simultaneously arrange that
	\begin{equation}\label{eq:delta condition}
	\frac{\rcapa_1(H\cap \tfrac 12 B,B)}{\rcapa_1(\tfrac 12 B,B)}>\delta
    \quad\textrm{and}\quad \frac{\rcapa_1(G\cap \tfrac 12 B,B)}{\rcapa_1(\tfrac 12 B,B)}
    >\delta.
	\end{equation}
	Then we can choose a set $F\subset B$ as given in Proposition \ref{prop:constructing the quasiconvex space}.
	In particular, we know that $x_0\in \fint F$.
	By \eqref{eq:capacity of F} and \eqref{eq:cap estimate} we know that
	\[
	\frac{\rcapa_1(\tfrac 12 B\setminus F,B)}{\rcapa_1(\tfrac 12 B,B)}
    \le \frac{9C'C_r\rcapa_1(W\cap B,2B)}{(C''C_d^2)^{-1}\rcapa_1(B,2B)}
	\le \delta.
	\]
	By \eqref{eq:cap of fine closure}, also 
	\[
	\frac{\rcapa_1(\overline{\tfrac 12 B\setminus F}^1,B)}{\rcapa_1(\tfrac 12 B,B)}
	\le \delta.
	\]
    Due to this combined with \eqref{eq:delta condition}, we find points
    $x\in \tfrac 12 B\cap H\cap \fint F$ and $y\in \tfrac 12 B\cap G\cap \fint F$.
	Now $\Theta^*(H,x)\ge 1/2$,
	whereas $\Theta(X\setminus F,x)=0$ by \eqref{eq:thinness and measure thinness},
	 and so $\frm(H\cap F)>0$.
     Similarly $\frm(G\cap F)>0$. This concludes Case a.
\bigskip

    \textbf{Case b: $\capa_1(\{x_0\})=0$.}
	By Lemma \ref{lem:open thin set}, we find an open set $W'$ containing
    $(\partial^*H\cap \partial^*G)\setminus \{x_0\}$
	and $X\setminus (G\cup H)$, such that $W'$ is also $1$-thin at $x_0$.
	Choose $0<R<\frac{1}{32} \diam X$ sufficiently small that
	\[
	\sup_{0<r\le R}\frac{\rcapa_1(W'\cap B(x_0,r),B(x_0,2r))}{\rcapa_1(B(x_0,r),B(x_0,2r))}
    <\min\left\{\frac{\delta}{18 C_d^2 C_r C' C''},\frac{1}{C_0 C_1C_d C_P}\right\},
	\]
	while with $B=B(x_0,R)$, by \eqref{eq:capacity of G and H},
	\begin{equation}\label{eq:delta condition 2}
	\frac{\rcapa_1(H\cap \tfrac 12 B,B)}{\rcapa_1(\tfrac 12 B,B)}>\delta
    \quad\textrm{and}\quad \frac{\rcapa_1(G\cap \tfrac 12 B,B)}{\rcapa_1(\tfrac 12 B,B)}
    >\delta.
	\end{equation}
Using the fact that $\capa_1(\{x_0\})=0$, by \eqref{eq:Cap and cap estimate}
and \eqref{eq:outer capacity}
we can choose $W:=W'\cup B(x_0,s)$ for sufficiently small $s>0$ that  still
\[
\frac{\rcapa_1(W\cap B,2B)}{\rcapa_1(B,2B)}
<\min\left\{\frac{\delta}{9 C_d^2 C_r C' C''},\frac{1}{C_0 C_1C_d C_P}\right\}.
\]
	Then we can choose a set $F\subset B$ as given in Proposition
    \ref{prop:constructing the quasiconvex space alternative}.
	By \eqref{eq:capacity of F 2} and \eqref{eq:cap estimate} we know that
	\[
	\frac{\rcapa_1(\tfrac 12 B\setminus F,B)}{\rcapa_1(\tfrac 12 B,B)}
    \le  \frac{9C'C_r\rcapa_1(W\cap B,2B)}{(C''C_d^2)^{-1}\rcapa_1(B,2B)}
	\le \delta.
	\]
	By \eqref{eq:cap of fine closure}, also 
	\[
	\frac{\rcapa_1(\overline{\tfrac 12 B\setminus F}^1,B)}{\rcapa_1(\tfrac 12 B,B)}
	\le \delta.
	\]
	Due to this combined with \eqref{eq:delta condition 2}, we find points
    $x\in \tfrac 12 B\cap H\cap \fint F$ and $y\in \tfrac 12 B\cap G\cap \fint F$.
	Now $\Theta^*(H,x)\ge 1/2$,
	whereas $\Theta(X\setminus F,x)=0$ by \eqref{eq:thinness and measure thinness},
	 and so $\frm(H\cap F)>0$.
     Similarly $\frm(G\cap F)>0$. This concludes Case b.
     \bigskip
	
	Denote upper density in the space $(F,d_M^F,\frm)$ by $\Theta_F^*$.
	Define two new sets $G',H'$ by
	\[
	G':=\{x\in F\colon \Theta_F^*(G,x)\ge 1/2\}
	\quad\textrm{and}\quad
	H':=\{x\in F\setminus G'\colon \Theta_F^*(H,x)\ge 1/2\}.
	\]
	Then $G'\cup H'=F$, and by Lebesgue's differentiation theorem, which is available since the measure $\frm$
	is locally doubling in $(F,d_M^F,\frm)$
    (see e.g.\ \cite[p. 77]{HKST15}), we have $\frm(G\Delta G')=0$ and $\frm(H\Delta H')=0$.
    Thus $\frm(G')>0$ and $\frm(H')>0$.

	By connectedness of $F$, there is a boundary point $y\in \partial_F H'=\partial_F G'$.
	Let $s>0$ be so small that $B(y,7s)\subset B(x_0,R)$.
	By the definition of $G'$ and $H'$,
	\[
	0<\frac{\frm(B_{F}(y,s)\cap G')}{\frm(B_{F}(y,s))}<1.
	\]
	By Proposition \ref{prop:constructing the quasiconvex space}
	we know that
	$\frm$ has the doubling property
    \[
    \frm(B_{F}(z,2r))\le C_1 C_d^2\frm(B_{F}(z,r))
    \]
    for every $z\in B(y,6s)$ and $0<r\le s$, and then also for 
    every $z\in B_F(y,6s)$ and $0<r\le s$.
	Moreover, $\overline{B}_F (y,6s)$ is compact.
	Then by Proposition \ref{prop:density points}, we find a point $z\in B_F(y,6s)$ such that
	\[
	\frac{1}{4 (C_1 C_d^2)^{12}}\le \liminf_{r\to 0}\frac{\frm(G'\cap B_{F}(z,r))}{\frm(B_{F}(z,r))}
	\le \limsup_{r\to 0}\frac{\frm(G'\cap B_{F}(z,r))}{\frm(B_{F}(z,r))}
	\le 1-\frac{1}{4 (C_1 C_d^2)^{12}},
	\]
	and the same inequalities for $H'$.
	By \eqref{eq:BF lower bound} and since $\frm(G\Delta G')=0$, we obtain
	\[
	\frac{1}{4 (C_1 C_d^2)^{13}}\le \liminf_{r\to 0}\frac{\frm(G\cap B(z,r))}{\frm(B(z,r))}
	\le \limsup_{r\to 0}\frac{\frm(G\cap B(z,r))}{\frm(B(z,r))}
	\le 1-\frac{1}{4  (C_1 C_d^2)^{13}},
	\]
	and the same inequalities for $H$.
	Thus we have found a point belonging to $\partial^*G\cap \partial^*H\cap F$,
	a contradiction with the construction of $F$ in both Cases a and b.
	Thus \eqref{eq:H and G} holds.
    By \eqref{eq:cap and H measure},
    in particular we get $\mathcal H(\partial^*H\cap \partial^*G)>0$,
	a contradiction with $P(V',X)=P(G,X)+P(H,X)$ by Lemma \ref{lem:measure theoretic boundaries}(i).
	Thus \eqref{eq:capacity of H} fails, and since we had
    $x_0\in \fint (G\cup H)$, we get $x_0\in \fint G$. Thus $G=\fint G$ and similarly $H=\fint H$.
	Now we have that $V'$ is the disjoint union of the $1$-finely open sets $G$ and $H$, a contradiction with the fact that
	$V'$ is connected in the $1$-fine topology.
	
	Thus $E$ is indecomposable.
    \bigskip

\textbf{Part 2.}

Now we consider the implication: if $E$ is indecomposable, then $\fint I_E$ is 1-finely connected.
First we make some general observations.

	Just as in \eqref{eq:countable union}, we have
	\begin{equation}\label{eq:countable union 2}
		\fint I_E=\bigcup_{j}V_j,
	\end{equation}
	where the $V_j$'s are pairwise disjoint $1$-finely open, $1$-finely connected sets.
    Consider two unions of these,
denoted by $V'=\bigcup_{j\in I}V_j$ and $V''=\bigcup_{j\in I'}V_j$, with $I\cap I'=\emptyset$.
 By Proposition \ref{prop:finite perimeter}, $V'$ and $V''$ are sets of finite perimeter.
Let $N^*$ be the $\mathcal H$-negligible exceptional set for $V' ,V''$ from the
definition of the two-sidedness property.

Suppose $x\in(\partial^* V'\cap \partial^* V'')\setminus (N\cup N^*)$.
From the two-sidedness property we get
$x\in I_{V'\cup V''}$, and so $x\in I_{\fint I_E}=I_E$
(recall Proposition \ref{prop:IE and OE}).
Since we had $x\notin N$, we get $x\in \fint I_E$ (recall \eqref{eq:definition of N}).
By \eqref{eq:countable union 2} it follows that $x$ must be in some set $V_k$ not contained in the unions
$V'$ and $V''$, and then by \eqref{eq:thinness and measure thinness} we get $x\in I_{V_k}$.
However, this contradicts the fact that simultaneously $x\in \partial^*V'$;
recall that $V',V'',V_k$ are pairwise disjoint.

We conclude that $\mathcal H(\partial^* V'\cap \partial^* V'')=0$.
By Lemma \ref{lem:measure theoretic boundaries} we know that 
\[
P(V'\cup V'',X)=P(V',X)+P(V'',X).
\]
Hence, if the union in \eqref{eq:countable union 2} contains more than one set $V_j$, we get
\begin{equation}\label{eq:V j}
P\Big(\bigcup_{j}V_j,X\Big)=P\Big(\bigcup_{j\neq \{k\}}V_j,X\Big)+P(V_k,X),
\end{equation}
showing that $\bigcup_{j}V_j$ is decomposable.
Thus, if $E$ is indecomposable, then $\fint I_E$ is $1$-finely connected.
\end{proof}

\begin{theorem}\label{thm:main}
	Suppose $X$ has the two-sidedness property. Let $E\subset X$ with $P(E,X)<\infty$.
	Then the $1$-finely connected components $\{V_j\}$ of $\fint I_E$
	are indecomposable and satisfy
    \[
    P(E,X)=\sum_{j}P(V_j,X).
    \]
    In other words, the sets $V_j$ are the essential connected components of $\fint I_E$.
\end{theorem}
\begin{proof}
By Theorem \ref{thm:iff}, each $V_j$ is indecomposable.
    Applying \eqref{eq:V j} iteratively, we get
    \[
    \sum_{j}P(V_j,X)\le P\Big(\bigcup_{j}V_j,X\Big),
    \]
    while the opposite inequality follows from \eqref{eq:perimeter of union}.
    Thus we have
    \[
    P(E,X)=P\Big(\bigcup_{j}V_j,X\Big) = \sum_{j}P(V_j,X).
    \]
\end{proof}

Theorem \ref{thm:main} gives a new proof of the existence of the essential connected components of $E$;
recall Theorem \ref{thm:essential connected components}.

\begin{proof}[Proof of Theorem \ref{thm:iff intro}]
This follows immediately from Theorems \ref{thm:iff} and \ref{thm:main}.
\end{proof}

\begin{example}\label{eq:three spider}
	Consider the ``$3$-spider'' $S_3$ defined by
	\[
	S_3:=\{o\}\sqcup (R_1\cup R_2\cup R_3),
	\quad\textrm{where }R_j:=\{j\}\times (0,\infty)\ \textrm{for }j=1,2,3.
	\]
	Here the point $o$ is the \emph{origin} of $S_3$, while $R_1,R_2,R_3$ are the \emph{rays} of $S_3$,
	and we identify $o$ with $(j,0)$ for each $j=1,2,3$.
	We consider the distance 
	\[
	d((j,s),(k,t)):=
	\begin{cases}
	|s-t|\quad &\textrm{if }j=k,\\
	t+s\quad &\textrm{if }j\neq k.
	\end{cases}
	\]
	Then $(S_3,d,\mathcal H^1)$ is a complete, doubling (in fact, Ahlfors $1$-regular) space supporting a $1$-Poincar\'e
	inequality.
	As noted in \cite[Lemma 1.27]{BPR}, this space is isotropic.
	However, it does not have the two-sidedness property;
	this can be seen from considering $E_1=R_1$, $E_2=R_2$, so that $o\in \partial^*E_1\cap \partial^*E_2$,
	but also $\Theta(X\setminus (E_1\cup E_2),o)=1/3>0$.
	
	Now consider $E:=E_1\cup E_2$.
	This is an indecomposable set (see \cite[Lemma 1.26]{BPR} for details).
	However, now $I_E=E\setminus \{0\}$ so that $\fint I_E=I_E$
	and the connected components of this set in the $1$-fine topology are $E_1$ and $E_2$.
	
	This shows that the assumption of the two-sidedness property in Theorem \ref{thm:main}
    cannot be removed.
\end{example}

\section{Fine domains}\label{sec:fine domains}

We apply the techniques developed in the previous sections to prove some results
on $1$-fine domains, that is, $1$-finely open sets that are connected in the $1$-fine topology.
Fine domains in the case $1<p<\infty$ have been studied in 
\cite{BB-UR,Lat}, see also references therein for previous works.

Recall the definition of the measure topology from Definition \ref{def:measure topology}.
In \cite[p. 225]{Lat}, it is asked
(in the case $1<p<\infty$) whether a $p$-fine domain is also a measure topology domain.
Now we can answer this question in the case $p=1$.

\begin{theorem}
Suppose $V'\subset X$ is a $1$-fine domain.
Then $V'$ is also a measure topology domain.
\end{theorem}
\begin{proof}
By \eqref{eq:thinness and measure thinness}, $V'$ is open in the measure topology.

To prove that $V'$ is connected in the measure topology, 
suppose instead that we had $V'\subset G\cup H$ for some sets $G=I_G,H=I_H$ with $G\cap H=\emptyset$,
	$\frm(G\cap V')>0$, $\frm(H\cap V')>0$.
    Using \eqref{eq:thinness and measure thinness} again, we can in fact assume that
    $G,H\subset V'$.
Next we follow verbatim the proof of Theorem \ref{thm:iff} Part 1  starting from line 10
(``Take $x_0\in G$...'') until line -7 (``Thus \eqref{eq:H and G} holds.'').
Indeed, \eqref{eq:H and G} is a contradiction with the fact that $x_0$ belongs to the
$1$-finely open set $V'\subset I_G\cup I_H$.
Then we conclude (just as in the last lines of the proof of Theorem \ref{thm:iff} Part 1)
 that \eqref{eq:capacity of H} fails, and since we had
    $x_0\in \fint (G\cup H)$, we get $x_0\in \fint G$. Thus $G=\fint G$ and similarly $H=\fint H$.
	Now we have that $V'$ is the disjoint union of the $1$-finely open sets $G$ and $H$, a contradiction with the fact that
	$V'$ is connected in the $1$-fine topology.
\end{proof}

The main result of \cite{Lat} states, in the case $1<p<\infty$,
that for a $p$-fine domain $V'\subset \R^n$ and a $\capa_p$-negligible set
$N\subset \R^n$, the set $V'\setminus N$ is also a $p$-fine domain.
Now we prove the analogous result in the case $p=1$.
Recall the definition of the $1$-base $b_1 A$ from Definition \ref{def:1 fine topology}.

\begin{theorem}\label{thm:fine domain}
Suppose $V'\subset X$ is a $1$-fine domain and suppose $N\subset V'$ with $\capa_1(N)=0$.
Then $V'\setminus N$ is also a $1$-fine domain.
\end{theorem}
\begin{proof}
By Lemma \ref{lem:invariance} and the discussion above it, we can assume that $X$ is geodesic.

By \eqref{eq:null sets of Hausdorff measure and capacity} and \eqref{eq:cap and H measure}, we have that
$V'\setminus N$ is also $1$-finely open. It remains to show that $V'\setminus N$ is $1$-finely connected.

Suppose instead that there are disjoint $1$-finely open $V_1,V_2\subset X$
with $V_1\cup V_2\supset V'\setminus N$.
If
\[
N= (N\setminus b_1 V_1)\cup (N\setminus b_1 V_2),
\]
then, noting that $b_1 N=\emptyset$ by \eqref{eq:cap and H measure},
we can define $V_1':=V_1\cup (N\setminus b_1 V_2)$
and
$V_2':=V_2\cup (N\setminus b_1 V_1)$, which are disjoint $1$-finely open sets
with $V'=V_1'\cup V_2'$, a contradiction.

Thus we can assume that there are $x_0\in N$ and $\delta>0$ with
\[
\limsup_{r\to 0}\frac{\rcapa_1(V_1\cap B(x_0,r),B(x_0,2r))}{\rcapa_1(B(x_0,r),B(x_0,2r))}>5 C_d^2 C' C'' \delta
\]
and
\[
\limsup_{r\to 0}\frac{\rcapa_1(V_2\cap B(x_0,r),B(x_0,2r))}{\rcapa_1(B(x_0,r),B(x_0,2r))}>5 C_d^2 C' C'' \delta.
\]
By Lemma \ref{lem:open thin set}, we find an open set $W'\supset (X\setminus V')\cup (N\setminus\{x_0\})$
that is $1$-thin at $x_0$.
Choose $0<R<\frac{1}{32} \diam X$ such that
\[
\sup_{0<r\le R}\frac{\rcapa_1(W'\cap B(x_0,r),B(x_0,2r))}{\rcapa_1(B(x_0,r),B(x_0,2r))}
<\min\Big\{\frac{\delta}{9C_d^2 C_r C' C''},\frac{1}{C_0 C_1 C_d C_P}\Big\};
\]
denoting $B=B(x_0,R)$, by Lemma \ref{lem:continuity argument} we can simultaneously get
\begin{equation}\label{eq:bigger than delta}
\frac{\min\{\rcapa_1(V_1\cap \tfrac 12 B,B),\rcapa_1(V_2\cap \tfrac 12 B,B)\}}{\rcapa_1(\tfrac 12 B,B)}>\delta.
\end{equation}
By \eqref{eq:null sets of Hausdorff measure and capacity} we have $\mathcal H(\{x_0\})=0$,
and so by \eqref{eq:cap and H measure}, \eqref{eq:outer capacity} we can
let $W:=W'\cup B(x_0,s)$ for sufficiently small $s>0$ that  still
\[
\frac{\rcapa_1(W\cap B,2B)}{\rcapa_1(B,2B)}
<\min\Big\{\frac{\delta}{9C_d^2 C_r C' C''},\frac{1}{C_0 C_1 C_d C_P}\Big\}.
\]
We find a connected set $F\subset B\setminus W\subset V'$
given by Proposition \ref{prop:constructing the quasiconvex space alternative};
we have
\begin{align*}
\rcapa_1(\tfrac 12 B\setminus F,B)
&\le 9C' C_r \rcapa_1(W\cap B,2B)\quad\textrm{by }\eqref{eq:capacity of F 2}\\
&<\frac{\delta}{C_d^2 C''}\rcapa_1(B,2B)\\
&\le \delta\rcapa_1(\tfrac 12 B,B)\quad\textrm{by }\eqref{eq:cap estimate}.
\end{align*}
We have $F\subset V_1\cup V_2$, where by \eqref{eq:bigger than delta} and the above,
$V_1\cap F$ and  $V_2\cap F$ are both nonempty.
By connectedness of $F$ (in the metric $d$-topology and hence also
in the $d_M^F$-topology), there is a boundary point $y\in \partial_F V_1$,
i.e.\ a point in the boundary of the set $V_1$ in the metric space
$(F,d^F_M)$.
Let $s>0$ be so small that $B(y,7s)\subset B(x_0,R)$.
By
\eqref{eq:thinness and measure thinness} and \eqref{eq:BF lower bound}, we have
\[
0<\frac{\frm(V_1\cap B_{F}(y,s))}{\frm(B_{F}(y,s))}<1.
\]
By Proposition \ref{prop:constructing the quasiconvex space}
we know that
$\frm$ has the doubling property
\[
\frm(B_{F}(z,2r))\le C_1 C_d^2\frm(B_{F}(z,r))
\]
for every $z\in B(y,6s)$ and $0<r\le s$, and then also
for every $z\in B_F(y,6s)$ and $0<r\le s$.
Moreover, $\overline{B}_F (y,6s)$ is compact.
Then by Proposition \ref{prop:density points}, we find a point $z\in B_F(y,6s)$
(in particular, $z\in F$)
such that
\[
\frac{1}{4 (C_1 C_d^2)^{12}}\le \liminf_{r\to 0}\frac{\frm(V_1\cap B_{F}(z,r))}{\frm(B_{F}(z,r))}
\le \limsup_{r\to 0}\frac{\frm(V_1\cap B_{F}(z,r))}{\frm(B_{F}(z,r))}
\le 1-\frac{1}{4 (C_1 C_d^2)^{12}}.
\]
By \eqref{eq:BF lower bound}, we obtain
\[
\frac{1}{4 (C_1 C_d^2)^{13}}\le \liminf_{r\to 0}\frac{\frm(V_1\cap B(z,r))}{\frm(B(z,r))}
\le \limsup_{r\to 0}\frac{\frm(V_1\cap B(z,r))}{\frm(B(z,r))}
\le 1-\frac{1}{4  (C_1 C_d^2)^{13}}.
\]
The same inequalities hold for $V_2$.
By \eqref{eq:thinness and measure thinness} we have
that $z\notin V_1\cup V_2$, a contradiction
with the fact that we had $F\subset V_1\cup V_2$.
Thus $V'\setminus N$ is $1$-finely connected.
\end{proof}

\begin{proposition}\label{prop:Gi to zero}
Consider a decreasing sequence of sets $G_i\subset X$, $G_{i+1}\subset G_i$ for all $i\in\N$,
with $\capa_1(G_i)\to 0$.
Let $V'\subset X$ be a $1$-fine domain.
Then there is an increasing sequence of $1$-fine domains $V_i\subset V'\setminus G_i$ such that
\[
\capa_1\left(V'\setminus \bigcup_{i=1}^{\infty}V_i\right)=0.
\]
\end{proposition}
\begin{proof}
The sequence $\overline{G_i}^1$ is also decreasing with
$\capa_1(\overline{G_i}^1)\to 0$ by \eqref{eq:cap of fine closure}. For every
$x\in V'\setminus \bigcap_{i\in\N}\overline{G_i}^1$, denote by $V_{x,i}$
the increasing sequence of
$1$-fine domains of $V'\setminus \overline{G_i}^1$ containing $x$.
By Proposition \ref{prop:open basis}, the $1$-finely topology is locally connected.
Thus the sets $\bigcup_{i\in\N}V_{x,i}$,
are $1$-fine domains that for different $x\in V'\setminus \bigcap_{i\in\N}\overline{G_i}^1$
either coincide or are pairwise disjoint, and
their union is $V'\setminus \bigcap_{i\in\N}\overline{G_i}^1$.
Denote these distinct sets by $\{V_k\}_k$
(necessarily a finite or countable collection, though we do not need this).
If for a given $k$ we have
\[
\capa_1(V'\setminus V_k)>0,
\]
then $V_k$ and $\bigcup_{j \neq k}V_j$ are two disjoint $1$-finely open sets
whose union is $V'\setminus \bigcap_{i\in\N}\overline{G_i}^1$.
However, by Theorem \ref{thm:fine domain} we know that
$V'\setminus \bigcap_{i\in\N}\overline{G_i}^1$ is $1$-finely connected, giving a contradiction.
Thus $\capa_1(V'\setminus V_k)=0$, which means that for some (in fact, all)
$x\in V'\setminus \bigcap_{i\in\N}\overline{G_i}^1$ we have that
$V_{x,i}\subset V'\setminus G_i$ is an increasing sequence of $1$-fine domains with
\[
\capa_1\left(V'\setminus \bigcup_{i=1}^{\infty}V_{x,i}\right)=0.
\]
\end{proof}

We note that with the theory that we have developed, it is also possible
to generalize the results of \cite{BB-UR} from $1<p<\infty$
to the full range $1\le p<\infty$, but we will leave this to a future work.

\section{Essential connected components via the metric topology}\label{sec:metric}

In this section we show that the
essential connected components of a set of finite perimeter
can also be obtained, though somewhat less elegantly, by means of the usual metric topology.

Note first that the metric topology is weaker than the $1$-fine topology, and so there are more
connected sets. Thus, instead of considering the connected components (in the metric topology)
of a set of finite perimeter $E$ (or, say, its measure-theoretic interior
$I_E$), we remove a set $G$ of small $1$-capacity,
and then we consider the connected components of $I_E\setminus G$.

More precisely, we define the \emph{quasi-components} of an arbitrary set $U$ in the following way.

\begin{definition}[$1$-quasi components]\label{def:quasi components}
Let $U\subset X$ and
consider a decreasing sequence of sets $G_i\subset X$, $G_{i+1}\subset G_i$ for all $i\in\N$,
with $\capa_1(G_i)\to 0$.
Then consider the connected components (in the metric topology) $\{W_{i,k}\}_k$ of
$U\setminus G_i$.
By reindexing if necessary, we can arrange that each $W_{i,k}$ is increasing with respect to $i$.
Then consider
\[
E_k:= \bigcup_{i=1}^\infty W_{i,k}.
\]
We say that such a set $E_k$ is minimal if for any other decreasing sequence
$G_i'$ with $\capa_1(G_i')\to 0$, and the corresponding sets $E_{k'}':=\lim_{i\to\infty}W_{i,k'}'$, for each $k$ there is $k'$ such that
\[
\capa_1(E_k\setminus E_{k'}')=0.
\]
If they exist, we call those minimal sets $E_k$ that
additionally satisfy $\capa_1(E_k)>0$ the \textbf{$1$-quasi-components} of $U$.
\end{definition}

Now we prove the main result of this section. For simplicity, we assume that $\frm(E)<\infty$.

\begin{theorem}\label{thm:quasicomponents}
Suppose $X$ has the two-sidedness property.
Let $E\subset X$ be such that $P(E,X)<\infty$ and $\frm(E)<\infty$.
Then the $1$-quasi-components of $I_E$ coincide, modulo $\capa_1$-negligible sets,
with the $1$-finely connected components of $\fint I_E$, and
thus also with the essential connected components of $E$.
\end{theorem}

\begin{proof}
By \eqref{eq:finely and quasiopen}, $I_E$ is the union of
$\fint I_E$ and a $\capa_1$-negligible set $N$.
Write $\fint I_E$ as the union of its $1$-finely connected components
\[
\fint I_E=\bigcup_{j}V_j.
\]
Consider a decreasing sequence $G_i'$ with $\capa_1(G_i')\to 0$,
and the connected components (in the metric topology) $\{W'_{i,k'}\}_{k'}$ of
$I_E\setminus G'_i$.
By Proposition \ref{prop:Gi to zero}, for each $j$
there is an increasing (w.r.t. $i$) sequence of $1$-fine domains $V_{i,j}\subset V_j\setminus G_i'$ such that
\[
\capa_1\left(V_{j}\setminus \bigcup_{i=1}^{\infty}V_{i,j}\right)=0.
\]
Each set $V_{i,j}$ is $1$-finely connected and thus also connected, and so we have that
$\bigcup_{i=1}^{\infty}V_{i,j}$ is contained in one of the sets $\bigcup_{i=1}^{\infty}W'_{i,k'}=E'_{k'}$.
Thus $V_j$ is contained (modulo $\capa_1$-negligible sets)
in one of the sets $E'_{k'}$.
In other words, the sets $E'_{k'}$ are (modulo $\capa_1$-negligible sets)
unions of the sets $V_j$ (plus possibly
a subset of $N$, which is however $\capa_1$-negligible).

Conversely, by Proposition \ref{prop:IE and OE} and \eqref{eq:finely and quasiopen} we know that every set 
$X\setminus (I_{V_j}\cup \partial^*V_j)$ is $1$-quasiopen,
and moreover by Theorem \ref{thm:iff intro}
and Lemma \ref{lem:measure theoretic boundaries} we have that
$\mathcal H(\partial^*V_j\cap \partial^*V_k)=0$ for all $j\neq k$.
There are open sets $H_{i,j}$
such that each $(I_{V_j}\cup \partial^*V_j)\setminus H_{i,j}$ is closed and $\capa_1(H_{i,j})<2^{-j}/i$.
Let $G_i$ be an open set containing the set
\[
N\cup \bigcup_{j}H_{i,j}\cup\bigcup_{j=N_i+1}^{\infty}V_j\cup 
\bigcup_{j\neq k}\partial^*V_j\cap \partial^*V_k;
\]
by selecting sufficiently large $N_i$, by Proposition \ref{prop:capacity sum}
we can ensure that $\capa_1(G_i)<1/i$.
Since $G_i \supset H_{i,j}$, for each $i\in\N$ we have 
$$(I_{V_j}\cup \partial^*V_j)\setminus G_{i} =   \bigr( (I_{V_j}\cup \partial^*V_j)\setminus H_{i,j} \bigl)\cap (X \setminus G_i).$$ It follows that the sets 
$\{(I_{V_j}\cup \partial^*V_j)\setminus G_i\}_{j=1}^{N_i}$ 
have three properties: firstly, they
are closed  (being the intersection of closed sets). They are also pairwise disjoint, because $G_i$ contains
$\bigcup_{j\neq k} \partial^*V_j\cap \partial^*V_k$ and because $I_{V_j}$ cannot intersect
$I_{V_k}\cup \partial^*V_k$ for $j\neq k$.
Thirdly, their union contains $I_E\setminus G_i$, because this equals
$\fint I_E\setminus G_i$ and $V_j\subset I_{V_j}$ by \eqref{eq:thinness and measure thinness}.
Hence, $\{(I_{V_j}\cup \partial^*V_j)\setminus G_i\}_{j=1}^{N_i}$ is a finite collection of pairwise disjoint compact sets
whose union contains $I_E\setminus G_i$.
Thus each connected component of $I_E\setminus G_i=\bigcup_{j}V_j\setminus G_i$
is contained in one $I_{V_j}\cup \partial^*V_j$,
and then in fact in one $V_j$.
Denote these connected components by $W_{i,k}$, and define
$E_k:=\bigcup_{i=1}^{\infty}W_{i,k}$.
Each $E_k$ is contained in one $V_j$.

Combining with the first part of the proof,
we conclude that the sets $E_k$ coincide (modulo $\capa_1$-negligible sets) with the sets $V_j$,
and are minimal in the sense of Definition \ref{def:quasi components}.
By Theorem \ref{thm:iff intro} the sets $V_j$ coincide
with
the essential connected components of $E$.
\end{proof}

\section{Cartan property}\label{sec:Cartan}
A well-known property in fine potential theory is the so-called Cartan property.
The property is well known in the case $1<p<\infty$, see e.g.\ \cite{MZ}.
In the metric measure space setting, a weak version of the property was shown in
\cite{BBL-WC} and a strong version in \cite{BBL-CCK}.
In the case $p=1$, a weak version of the property was shown by the second author in \cite{L-WC}.

In this section we prove a strong version in the case $p=1$;
the difference with the weak version is that there is only
superminimizer function $u=\ch_E$,
whereas in the weak version one needs a finite number (typically two) superminimizer functions.
In Proposition \ref{prop:constructing the quasiconvex space}, it was crucial to have the strong version.

We will rely on the following  coarea inequality and ``pasting result'' for sets of finite perimeter.

\begin{proposition}[{\cite[Proposition 5.1]{A1}}]\label{prop:coarea}
For any ball $B(x,R)$ and Borel set $A\subset X$, we have
\[
\int_0^R \mathcal H(\partial B(x,r)\cap A)\,\dd r\le 2\frm(B(x,R)\cap A).
\]
\end{proposition}

\begin{lemma}[{\cite[Lemma 3.5]{A2}}]\label{lem:E minus ball}
 Let $E\subset X$ be a set of finite perimeter in $X$ and let $x\in X$. For a.e.\ $r>0$ the set $E\setminus B(x,r)$
 has finite perimeter in $X$ and satisfies 
 \[
 P(E\setminus B(x,r), \partial B(x,r))\le \frac{d}{ds}\frm (B(x,s)\cap E)\Big|_{s=r}.
 \]
\end{lemma}

Recall that by our standing assumptions, we have $\frm(X)>0$ (possibly $\infty$),
and $\frm(\{x\})=0$ for every $x\in X$ (see e.g.\ \cite[Corollary 3.9]{BB}).
We restate the theorem here.

\begin{theorem}
	Let $W\subset X$ be open and let $x\in X\setminus W$ be such that $W$
	is $1$-thin at $x$.
	Choose $0<R<\frac{1}{32} \diam X$ sufficiently small that
	\begin{equation}\label{eq:space measure2}
	\frac{\frm(B(x,2R))}{\frm(X)}<C_0^{-1}
	\end{equation}
and
	\begin{equation}\label{eq:small capacity2}
	\sup_{0<r\le R}\frac{\rcapa_1(W\cap B(x,r),B(x,2r))}{\rcapa_1(B(x,r),B(x,2r))}<C_0^{-1},
	\end{equation}
	where
	\[
	C_0:=(20C_{SP} C_2 C_S C_d^{1+2\lceil\log_2(4\lambda)\rceil})^{2Q}.
	\]
	Denote $B:=B(x,R)$.
	Let $\ch_V=\ch^\wedge_{V}$ be a solution of the
	$\mathcal K_{W\cap B,0}(4B)$-obstacle problem (as guaranteed by Lemma \ref{lem:solutions from capacity}).
	Then $V$ is open, contained in $2B$,
	$\ch_V^{\wedge}=1$ in $W\cap B$,
	$\ch_V^{\vee}(x)=0$, and
	\begin{equation}\label{eq:V cap density zero2}
	\lim_{r\to 0}\frac{\rcapa_1(V\cap B(x,r),B(x,2r))}{\rcapa_1(B(x,r),B(x,2r))}=0.
	\end{equation}
\end{theorem}

\begin{proof}
	Let $\ch_V$ be a solution to the
	$\mathcal K_{W\cap B,0}(4B)$-obstacle problem, given by Lemma \ref{lem:solutions from capacity}.
	By Lemma \ref{lem:smallness in annuli} we know that
	for all
	$y\in 3 B\setminus 2B$,
	\begin{align*}
	\ch_V^{\vee}(y)
    \le C_2 R \frac{\rcapa_1(W\cap B,4B)}{\frm(B)}
    &\le C_2 R \frac{\rcapa_1(W\cap B,2B)}{\frm(B)}\\
    &\le C_2 C_d \frac{\rcapa_1(W\cap B,2B)}{\rcapa_1(B,2B)}\quad \textrm{by }\eqref{eq:cap estimate}\\
    &<1,
	\end{align*}
	and so $\ch_V^{\vee}=0$ in $3 B\setminus 2B$.
	Then
	in fact $\ch_V^{\vee}=0$ in $4B\setminus 2B$,
	because else we could remove the part of $V$ inside $4B\setminus 3B$
	to decrease $P(V,X)$.
	Thus $\ch_V^{\vee}=0$ in $X\setminus 2B$.

    In the proof, we will often use the second inequality of \eqref{eq:cap estimate},
    and we will do so without referring to it each time.
	
	By Theorem \ref{thm:superminimizers are lsc}, the function $\ch_V^{\wedge}$
	is lower semicontinuous on $4B$,
	and by redefining $V$ in a set of measure zero such that $\ch_V=\ch^\wedge_{V}$,
	we find that $V\subset 2B$ is open.
    \bigskip

By Lemma \ref{lem:solutions from capacity} and the assumption \eqref{eq:small capacity2}, we have
\begin{equation}\label{eq:perimeter V total}
P(V,X)\le C_0^{-1}\rcapa_1(B(x,R),B(x,2R))\le C_dC_0^{-1}\frac{\frm(B(x,R))}{R}.
\end{equation}
By \eqref{eq:isop inequality with zero boundary values}, we also have
\begin{equation}\label{eq:isoperimetric ineq}
\frm(V)\le 2C_S RP(V,X)\le 2 C_S C_d C_0^{-1} \frm(B(x,R)).
\end{equation}

To prove the last claim \eqref{eq:V cap density zero2}, we consider two cases that will be further split into subcases.
In each case, the idea will be to assume that \eqref{eq:V cap density zero2} fails, and derive a contradiction with the fact
that $V$ is a solution.
\bigskip

\textbf{Case 1.}
Suppose $\Theta^*(V,x)>0$;
in other words, we have
\[
\limsup_{r\to 0}\frac{\frm(V\cap B(x,r))}{\frm(B(x,r))}=:\kappa>0.
\]

\textbf{Case 1a.}
Suppose that $\kappa> C_0^{-1/2}$.

We find $0<r<R/2$ such that
\[
\frac{\frm(V\cap B(x,r/4\lambda))}{\frm(B(x,r/4\lambda))}>C_0^{-1/2}.
\]
Choose the smallest $j\in\N$ for which (this is possible by \eqref{eq:space measure2})
\[
\frac{\frm(V\cap B(x,(4\lambda)^{j-1} r))}{\frm(B(x,(4\lambda)^{j-1} r))}\le C_0^{-1/2}.
\]
Denoting $(4\lambda)^{j-1} r=s$, we have
\begin{equation}\label{eq:V and C0}
\frac{\frm(V\cap B(x,s))}{\frm(B(x,s))}\le C_0^{-1/2}
\end{equation}
and
\[
C_0^{-1/2} < \frac{\frm(V\cap B(x,s/4\lambda))}{\frm(B(x,s/4\lambda))}
\le C_d^{\lceil\log_2(4\lambda)\rceil}C_0^{-1/2}.
\]
Note that if we had $s\ge R$, it would follow that
\begin{align*}
\frm(V)
\ge \frm(V\cap B(x,s/4\lambda))
&\ge C_0^{-1/2} \frm(B(x,s/4\lambda))\\
&\ge C_0^{-1/2} C_d^{-\lceil\log_2(4\lambda)\rceil} \frm(B(x,s))\\
&\ge C_0^{-1/2} C_d^{-\lceil\log_2(4\lambda)\rceil} \frm(B(x,R)),
\end{align*}
contradicting \eqref{eq:isoperimetric ineq} by the choice of $C_0$.
Thus $s<R$.

Now by the relative isoperimetric inequality \eqref{eq:relative isoperimetric inequality},
\begin{align*}
	C_0^{-1/2}\frm(B(x,s/(4\lambda)))
	&\le \frm(V\cap B(x,s/(4\lambda)))\\
	&\le 2C_{SP}\frac{s}{4\lambda}\left(\frac{\frm(V \cap  B(x,s/4\lambda))}{\frm(B(x,s/4\lambda))}\right)^{1/Q} P(V,B(x,s/2))\\
	&\le C_{SP}\frac{s}{\lambda} C_0^{-1/2Q} C_d^{\lceil\log_2(4\lambda)\rceil/Q}P(V,B(x,s/2)).
\end{align*}
Hence
\begin{align*}
P(V,B(x,s/2))
&\ge \lambda C_0^{-1/2+1/2Q}C_{SP}^{-1}C_d^{-\lceil\log_2(4\lambda)\rceil(1+1/Q)}\frac{\frm(B(x,s))}{s}\\
&>4C_d C_0^{-1/2}\frac{\frm(B(x,s))}{s}
\end{align*}
by the choice of $C_0$.
From Lemma \ref{lem:E minus ball} we find $s'\in (s/2,s)$ such that
\[
P(V\setminus B(x,s'), \partial B(x,s'))
\le 2\frac{\frm(V\cap B(x,s))}{s}
\le 2C_0^{-1/2}\frac{\frm(B(x,s))}{s}\quad\textrm{by }\eqref{eq:V and C0}.
\]
Moreover, by Lemma \ref{lem:solutions from capacity} and the assumption \eqref{eq:small capacity2},
we can find $D\subset B(x,2s)$ containing $B(x,s)\cap W$ such that
\[
	P(D,X)\le C_0^{-1} \rcapa_1(B(x,s),B(x,2s)) \le C_0^{-1}C_d \frac{\frm(B(x,s))}{s}.
\]
We wish to compare the solution $V$ to the competitor $D\cup (V\setminus B(x,s'))$.
Using first the measure property of the perimeter and then also the subadditivity
\eqref{eq:perimeter subadditivity}, we have
\begin{align*}
&P(D\cup (V\setminus B(x,s')),X)-P(V,X)\\
&\quad=P(D\cup (V\setminus B(x,s')),B(x,s'))-P(V,B(x,s'))\\
&\quad\quad +P(D\cup (V\setminus B(x,s')),\partial B(x,s'))-P(V,\partial B(x,s'))\\
&\quad\qquad +P(D\cup (V\setminus B(x,s')),X\setminus \overline{B(x,s')})-P(V,X\setminus \overline{B(x,s')})\\
&\quad\le P(D\cup (V\setminus B(x,s')),B(x,s'))-P(V,B(x,s/2))\\
&\quad\quad +P(D\cup (V\setminus B(x,s')),\partial B(x,s'))\\
&\quad\qquad +P(D,X\setminus \overline{B(x,s')})+P(V\setminus B(x,s'),X\setminus \overline{B(x,s')})-P(V,X\setminus \overline{B(x,s')})\\
&\quad= P(D,B(x,s'))-P(V,B(x,s/2))\\
&\quad\quad +P(D\cup (V\setminus B(x,s')),\partial B(x,s'))\\
&\quad\qquad +P(D,X\setminus \overline{B(x,s')})\\
&\quad\le  P(D,B(x,s'))-P(V,B(x,s/2))\\
&\quad\quad +P(D,\partial B(x,s'))+P(V\setminus B(x,s'),\partial B(x,s'))\\
&\quad\qquad +P(D,X\setminus \overline{B(x,s')})\\
&\quad= P(D,X)-P(V,B(x,s/2))+P(V\setminus B(x,s'),\partial B(x,s'))
\end{align*}
by using the measure property again.

We record this:
\begin{equation}\label{eq:set surgery}
\begin{split}
&P(D\cup (V\setminus B(x,s')),X)-P(V,X)\\
&\quad \le P(D,X)-P(V,B(x,s/2))+P(V\setminus B(x,s'),\partial B(x,s')).
\end{split}
\end{equation}
Combining with the previous estimates, we get
\begin{align*}
P(D\cup (V\setminus B(x,s')),X)-P(V,X)
&\le 3 C_0^{-1/2}C_d \frac{\frm(B(x,s))}{s}-4C_d C_0^{-1/2}\frac{\frm(B(x,s))}{s}\\
& < 0.
\end{align*}
This is a contradiction with the fact that $V$ is a solution
of the $\mathcal K_{W\cap B,0}(4B)$-obstacle problem.
\bigskip

\textbf{Case 1b.}
Suppose that $\kappa\le C_0^{-1/2}$.

\textbf{Case 1b(i).} Suppose that
\begin{equation}\label{eq:first assumption}
	0=\liminf_{r\to 0}\frac{\frm(V\cap B(x,r))}{\frm(B(x,r))}
	<\limsup_{r\to 0}\frac{\frm(V\cap B(x,r))}{\frm(B(x,r))}.
\end{equation}
Recall that the value of the ``$\limsup$'' is denoted by $\kappa$.
Define $\delta>0$ by
\begin{equation}\label{eq:choice of delta}
	\delta:=\min\left\{\frac{\kappa}{C_d^{\lceil \log_2 (4\lambda)\rceil }},\frac{1}{ 
		 (4C_{SP})^{Q} C_d^{Q+\lceil\log_2(4\lambda)\rceil(1+Q)}} \right\}.
\end{equation}
Fix $r>0$ sufficiently small that
\begin{equation}\label{eq:smallness of capacity around x}
	\sup_{0<r'\le r}\frac{\rcapa_1(W\cap B(x,r'),B(x,2r'))}{\rcapa_1(B(x,r'),B(x,2r'))}
	<\delta,
\end{equation}
while
\[
 \frac{\frm(V\cap B(x,r))}{\frm(B(x,r))}\le \frac{\delta}{2}.
\]
Since $\delta \le \kappa/C_d^{\lceil \log_2 (4\lambda)\rceil }$, we can choose the smallest $j\in\N$ such that
\[
\frac{\frm(V\cap B(x,(4\lambda)^{-j} r))}{\frm(B(x,(4\lambda)^{-j} r))}> \frac{\delta}{2}.
\]
Let $s:=(4\lambda)^{-j+1} r$. Then
\[
 \frac{\frm(V\cap B(x,s))}{\frm(B(x,s))}
\le \frac{\delta}{2}
\]
and
\[
\frac{\delta}{2} < \frac{\frm(V\cap B(x,s/4\lambda))}{\frm(B(x,s/4\lambda))}
\le C_d^{\lceil \log_2 (4\lambda)\rceil } \frac{\delta}{2}.
\]
Thus by the relative isoperimetric inequality \eqref{eq:relative isoperimetric inequality},
\begin{align*}
	\frac{\delta}{2} \frm(B(x,s/4\lambda))
	&\le \frm(V\cap B(x,s/4\lambda))\\
	&\le 2C_{SP}\frac{s}{4\lambda}\left(\frac{\frm(V\cap B(x,s/4\lambda))}{\frm(B(x,s/4\lambda))}\right)^{1/Q} P(V,B(x,s/2))\\
	&\le 2C_{SP}\frac{s}{4\lambda}\delta^{1/Q} C_d^{\lceil \log_2 (4\lambda)\rceil /Q}  P(V,B(x,s/2)),
\end{align*}
and so
\[
P(V,B(x,s/2))\ge 4C_d \delta\frac{\frm(B(x,s))}{s}
\]
by the choice of $\delta$ in \eqref{eq:choice of delta}.

From Lemma \ref{lem:E minus ball} we find $s'\in (s/2,s)$ such that
\[
P(V\setminus B(x,s'), \partial B(x,s'))
\le 2\frac{\frm(V\cap B(x,s))}{s}
\le \delta \frac{\frm(B(x,s))}{ s}.
\]
Moreover, by \eqref{eq:smallness of capacity around x}
and Lemma \ref{lem:solutions from capacity}
we can find $D\subset B(x,2s)$ containing $B(x,s)\cap W$ such that
\[
P(D,X)\le \rcapa_1(W\cap B(x,s),B(x,2s)) 
\le C_d \delta \frac{\frm(B(x,s))}{s}.
\]
By \eqref{eq:set surgery}, we have
\begin{align*}
		&P(D\cup (V\setminus B(x,s')),X)-P(V,X)\\
		&\quad \le P(D,X)-P(V,B(x,s/2))+P(V\setminus B(x,s'),\partial B(x,s'))\\
		&\quad \le  C_d \delta \frac{\frm(B(x,s))}{s} 
		-4C_d \delta\frac{\frm(B(x,s))}{s}
		+\delta  \frac{\frm(B(x,s))}{ s}\\
		&\quad <0\quad\textrm{by }\eqref{eq:choice of delta}.
\end{align*}
This is a contradiction with the fact that $V$ is a solution
of the $\mathcal K_{W\cap B,0}(4B)$-obstacle problem.
\bigskip

\textbf{Case 1b(ii).} Alternatively, suppose that
\[
\limsup_{r\to 0}\frac{\frm(V\cap B(x,r))}{\frm(B(x,r))}\le C_0^{-1/2}
\]
while
\begin{equation}\label{eq:liminf eps}
\liminf_{r\to 0}\frac{\frm(V\cap B(x,r))}{\frm(B(x,r))}= \eps>0.
\end{equation}
Choose $0<r\le R$ such that
\begin{equation}\label{eq:sup eps}
\sup_{0<r'\le r}\frac{\rcapa_1(W\cap B(x,r'),B(x,2r'))}{\rcapa_1(B(x,r'),B(x,2r'))}<\frac{\eps}{C_d}
\end{equation}
and 
\begin{equation}\label{eq:alternative assumption}
\eps/2
\le \inf_{0<r'\le r}\frac{\frm(V\cap B(x,r'))}{\frm(B(x,r'))}
\le \sup_{0<r'\le r}\frac{\frm(V\cap B(x,r'))}{\frm(B(x,r'))}<2C_0^{-1/2}.
\end{equation}
We find some radius $0<s\le r$ for which
\begin{equation}\label{eq:doubling type}
\frac{\frm(V\cap B(x,s))}{\frm(V\cap B(x,s/4\lambda))}\le 2C_d^{\lceil \log_2 (4\lambda)\rceil };
\end{equation}
if such a radius did not exist, then we would have for all $0<s\le r$
\[
\frac{\frm(V\cap B(x,s))}{\frm(V\cap B(x,s/4\lambda))}> 2C_d^{\lceil \log_2 (4\lambda)\rceil }
\ge 2\frac{\frm(B(x,s))}{\frm(B(x,s/4\lambda))},
\]
and then
\[
\frac{\frm(V\cap B(x,r))}{\frm(B(x,r))}
>2\frac{\frm(V\cap B(x,r/4\lambda))}{\frm(B(x,r/4\lambda))}
>\ldots
>2^j\frac{\frm(V\cap B(x,r/(4\lambda)^j))}{\frm(B(x,r/(4\lambda)^j))},
\]
where the last quantity tends to $\infty$ as $j\to\infty$ due to \eqref{eq:liminf eps}, a contradiction.
Thus \eqref{eq:doubling type} holds and by the relative isoperimetric inequality \eqref{eq:relative isoperimetric inequality},
\begin{align*}
    \frm(V\cap B(x,s/4\lambda))
	&\le 2C_{SP}\frac{s}{4\lambda}\left(\frac{\frm(V\cap B(x,s/4\lambda))}{\frm(B(x,s/4\lambda))}\right)^{1/Q} P(V,B(x,s/2))\\
	&\le 2^{1+1/Q}C_{SP}\frac{s}{4\lambda} C_0^{-1/2Q} P(V,B(x,s/2)),
\end{align*}
and so
\[
P(V,B(x,s/2))\ge 10C_d^{\lceil \log_2 (4\lambda)\rceil } \frac{\frm(V\cap B(x,s/4\lambda))}{s}
\]
by the choice of $C_0$.
From Lemma \ref{lem:E minus ball} we find $s'\in (s/2,s)$ such that
\[
P(V\setminus B(x,s'), \partial B(x,s'))
\le 2\frac{\frm(V\cap B(x,s))}{s}
\le 4C_d^{\lceil \log_2 (4\lambda)\rceil}\frac{\frm(V\cap B(x,s/4\lambda))}{s}
\]
by \eqref{eq:doubling type}.
Moreover, by \eqref{eq:sup eps} and Lemma \ref{lem:solutions from capacity}
we can find $D\subset B(x,2s)$ containing $W\cap B(x,s)$ such that
\[
P(D,X)\le \rcapa_1(W\cap B(x,s),B(x,2s)) \le \eps\frac{\frm(B(x,s))}{s}.
\]
By \eqref{eq:set surgery}, we have
\begin{align*}
		&P(D\cup (V\setminus B(x,s')),X)-P(V,X)\\
		&\quad \le P(D,X)-P(V,B(x,s/2))+P(V\setminus B(x,s'),\partial B(x,s'))\\
		&\quad \le \eps\frac{\frm(B(x,s))}{s}
		-10C_d^{\lceil \log_2 (4\lambda)\rceil }\frac{\frm(V\cap B(x,s/4\lambda))}{s}\\
		&\qquad\qquad +4C_d^{\lceil \log_2 (4\lambda)\rceil }\frac{\frm(V\cap B(x,s/4\lambda))}{s}\\
		&\quad \le \eps\frac{\frm(B(x,s))}{s}
		-6C_d^{\lceil \log_2 (4\lambda)\rceil }\frac{\frm(V\cap B(x,s/4\lambda))}{s}\\
		&\quad <-2\eps\frac{\frm(B(x,s))}{s} \quad\textrm{by }\eqref{eq:alternative assumption}\\
        &\quad <0.
\end{align*}

This is a contradiction with the fact that $V$ is a solution
of the $\mathcal K_{W\cap B,0}(4B)$-obstacle problem.
\bigskip

\textbf{Case 2.}
We have now ruled out the case $\Theta^*(V,x)>0$, and so we can assume
that $\Theta(V,x)=0$.

Suppose, toward another contradiction, that we had
\[
\limsup_{r\to 0}\frac{P(V,B(x,r))}{\frm(B(x,r))/r}=:\delta>0.
\]
Then choose $0<s<R$ such that
\begin{equation}\label{eq:small measure density of W}
\sup_{0<r\le s}\frac{\frm(V\cap B(x,r))}{\frm(B(x,r))}<\frac{\delta}{5C_d}
\end{equation}
and 
\[
\sup_{0<r \le s}\frac{\rcapa_1(W\cap B(x,r),B(x,2r))}{\rcapa_1(B(x,r),B(x,2r))}<\frac{\delta}{5C_d^2}
\]
and also
\[
\frac{P(V,B(x,s/2))}{\frm(B(x,s/2))/s}>\delta.
\]
From Lemma \ref{lem:E minus ball}
and \eqref{eq:small measure density of W}, we find $s'\in (s/2,s)$ such that
\[
P(V\setminus B(x,s'), \partial B(x,s'))
\le 2\frac{\frm(V\cap B(x,s))}{s}
<\frac{2\delta}{5C_d}\frac{\frm(B(x,s))}{s}.
\]
Moreover, by Lemma \ref{lem:solutions from capacity}
we can find $D\subset B(x,2s)$ containing $W\cap B(x,s)$ such that
\[
P(D,X) \le \rcapa_1(W\cap B(x,s),B(x,2s))
\le \frac{\delta}{5C_d}\frac{\frm(B(x,s))}{s}.
\]
By \eqref{eq:set surgery}, we have
\begin{align*}
	&P(D\cup (V\setminus B(x,s')),X)-P(V,X)\\
	&\quad \le P(D,X)-P(V,B(x,s/2))+P(V\setminus B(x,s'),\partial B(x,s'))\\
	&\quad \le \frac{\delta}{5C_d}\frac{\frm(B(x,s))}{s}
	-\frac{\delta}{C_d}\frac{\frm(B(x,s))}{s}
	 +\frac{2\delta}{5C_d}\frac{\frm(B(x,s))}{s}\\
	&\quad <0.
\end{align*}
This contradicts the fact that $V$ is a solution.

We conclude that
\[
\lim_{r\to 0}\frac{P(V,B(x,r))}{\frm(B(x,r))/r}=0.
\]
By the doubling property of $\frm$, it follows that
\[
\lim_{r\to 0}\frac{P(V,B(x,2r))}{\frm(B(x,r))/r}=0.
\]
By Proposition \ref{prop:coarea}, for every $r>0$ we find $s_r\in [r,2r)$ such that
\[
\mathcal H(\partial (B(x,s_r)\cap (I_V\cup \partial^* V))
\le 2\frac{\frm((I_V\cup \partial^* V)\cap B(x,2r))}{r}
= 2\frac{\frm(V\cap B(x,2r))}{r}.
\]
It is easy to check that
\begin{equation}\label{eq:ms th boundaries}
\partial^*(V\cap B(x,s_r))\subset (\partial^* V\cap B(x,s_r))\cup (\partial B(x,s_r)\cap (I_V\cup \partial^* V)).
\end{equation}
By a combination of Theorem \ref{thm:Federers characterization} and \eqref{eq:def of theta},
\begin{align*}
&P(V\cap B(x,s_r),X)\\
&\quad \le C_d \mathcal H(\partial^*(V\cap B(x,s_r)))\\
&\quad \le C_d\mathcal H(\partial^* V\cap B(x,2r))+C_d\mathcal H(\partial B(x,s_r)\cap (I_V\cup \partial^* V))
\quad\textrm{by }\eqref{eq:ms th boundaries}\\
&\quad \le \alpha^{-1} C_dP(V,B(x,2r))+2C_d\frac{\frm(V\cap B(x,2r))}{r},
\end{align*}
and then, since $\Theta(V,x)=0$, we get
\[
\lim_{r\to 0}\frac{\rcapa^{\vee}_{\BV}(V\cap B(x,s_r),B(x,2r))}{\frm(B(x,r))/r}
\le \lim_{r\to 0}\frac{P(V\cap B(x,s_r),X)}{\frm(B(x,r))/r}=0.
\]
By \eqref{eq:variational one and BV capacity}
and \eqref{eq:cap estimate}, we obtain
\[
\lim_{r\to 0}\frac{\rcapa_1(V\cap B(x,r),B(x,2r))}{\rcapa_1(B(x,r),B(x,2r))}=0.
\]
This completes the proof of the theorem.
\end{proof}

\noindent Addresses:
\bigskip

\textbf{Paolo Bonicatto}, 
Dipartimento di Matematica, Università di Trento,
Via Sommarive 14, 38123 Trento, Italy,
{\tt paolo.bonicatto@unitn.it}
\bigskip

\textbf{Panu Lahti}, Academy of Mathematics and Systems Science, Chinese Academy of Sciences,
	Beijing 100190, PR China, {\tt panulahti@amss.ac.cn}
\bigskip

\textbf{Enrico Pasqualetto}, Department of Mathematics and Statistics,
P.O.\ Box 35 (MaD), FI-40014 University of Jyväskylä, Finland,
{\tt enrico.e.pasqualetto@jyu.fi}


\begin{thebibliography}{ACMM}

\bibitem{A1}L. Ambrosio,
\textit{Fine properties of sets of finite perimeter in doubling metric measure spaces},
Calculus of variations, nonsmooth analysis and related topics.
Set-Valued Anal. 10 (2002), no. 2-3, 111--128.

\bibitem{A2}L. Ambrosio,
\textit{Some fine properties of sets of finite perimeter in Ahlfors regular metric measure spaces},
Adv. Math. 159 (2001), no. 1, 51--67.

\bibitem{ACMM}L. Ambrosio, V. Caselles, S. Masnou, and J.-M. Morel,
\textit{Connected components of sets of finite perimeter and applications to image processing},
J. Eur. Math. Soc. (JEMS) 3 (2001), no. 1, 39--92.

\bibitem{AFP}L. Ambrosio, N. Fusco, and D. Pallara,
\textit{Functions of bounded variation and free discontinuity problems.}
Oxford Mathematical Monographs. The Clarendon Press, Oxford University Press, New York, 2000.

\bibitem{AMP}L. Ambrosio, M. Miranda, Jr., and D. Pallara,
\textit{Special functions of bounded variation in doubling metric measure spaces},
Calculus of variations: topics from the mathematical heritage of E. De Giorgi, 1--45,
Quad. Mat., 14, Dept. Math., Seconda Univ. Napoli, Caserta, 2004.

\bibitem{AS}L. Ambrosio and M. Scienza,
\textit{Locality of the perimeter in Carnot groups and chain rule},
Annali di Matematica Pura ed Applicata 189 (2010), no.4, 661--678.

\bibitem{BB-UR}A. Bj\"orn and J. Bj\"orn,
\textit{A uniqueness result for functions with zero fine gradient on quasiconnected and finely connected sets,}
Ann. Sc. Norm. Super. Pisa Cl. Sci. (5) 21 (2020), 293--301.

\bibitem{BB}A. Bj\"orn and J. Bj\"orn,
\textit{Nonlinear potential theory on metric spaces},
EMS Tracts in Mathematics, 17. European Mathematical Society (EMS), Z\"urich, 2011. xii+403 pp.

\bibitem{BBL-CCK}A. Bj\"orn, J. Bj\"orn, and V. Latvala,
\textit{The Cartan, Choquet and Kellogg properties for the fine topology on metric spaces},
J. Anal. Math. 135 (2018), no. 1, 59--83.

\bibitem{BBL-WC}A. Bj\"orn, J. Bj\"orn, and V. Latvala,
\textit{The weak Cartan property for the p-fine topology on metric spaces},
Indiana Univ. Math. J. 64 (2015), no. 3, 915--941.

\bibitem{BPR}P. Bonicatto, E. Pasqualetto, and T. Rajala,
\textit{Indecomposable sets of finite perimeter in doubling metric measure spaces},
Calc. Var. Partial Differential Equations 59 (2020), no. 2, Paper No. 63, 39 pp.

\bibitem{CCDM}F. Cagnetti, M. Colombo, G. De Philippis, and F. Maggi,
\textit{Essential connectedness and the rigidity problem for Gaussian symmetrization},
J. Eur. Math. Soc. (JEMS) 19 (2017), no. 2, 395--439.

\bibitem{CCDM2}F. Cagnetti, M. Colombo, G. De Philippis, and F. Maggi,
\textit{Rigidity of equality cases in Steiner's perimeter inequality},
Anal. PDE 7 (2014), no. 7, 1535–-1593.

\bibitem{DCNRV}A. Di Castro, M. Novaga, B. Ruffini, and E. Valdinoci,
\textit{Nonlocal quantitative isoperimetric inequalities},
Calc. Var. Partial Differential Equations 54 (2015), no. 3, 2421--2464.

\bibitem{Fed}H. Federer,
\textit{Geometric measure theory},
Die Grundlehren der mathematischen Wissenschaften, Band 153 Springer-Verlag New York Inc., New York 1969 xiv+676 pp.

\bibitem{GMS}M. Giaquinta, G. Modica, and J. Sou\v{c}ek,
\textit{Cartesian currents in the calculus of variations. I.
Cartesian currents},
Ergeb. Math. Grenzgeb. (3), 37 [Results in Mathematics and Related Areas. 3rd Series. A Series of Modern Surveys in Mathematics]
Springer-Verlag, Berlin, 1998. xxiv+711 pp.

\bibitem{Hj}P. Haj\l{}asz,
\textit{Sobolev spaces on metric-measure spaces.}
Heat kernels and analysis on manifolds, graphs, and metric spaces (Paris, 2002), 173--218,
Contemp. Math., 338, Amer. Math. Soc., Providence, RI, 2003.

\bibitem{HaKi}H. Hakkarainen and J. Kinnunen,
\textit{The BV-capacity in metric spaces},
Manuscripta Math. 132 (2010), no. 1-2, 51--73.

\bibitem{HK}J. Heinonen and P. Koskela,
\textit{Quasiconformal maps in metric spaces with controlled geometry},
Acta Math. 181 (1998), no. 1, 1--61.

\bibitem{HKST15}J. Heinonen, P. Koskela, N. Shanmugalingam, and J. Tyson,
\textit{Sobolev spaces on metric measure spaces.
	An approach based on upper gradients},
New Mathematical Monographs, 27. Cambridge University Press, Cambridge, 2015. xii+434 pp.

\bibitem{KKST}J. Kinnunen, R. Korte, N. Shanmugalingam, and H. Tuominen,
\textit{Lebesgue points and capacities via the boxing inequality in metric spaces},
Indiana Univ. Math. J. 57 (2008), no. 1, 401--430.

\bibitem{KoLa}R. Korte and P. Lahti,
\textit{Relative isoperimetric inequalities and sufficient conditions for finite perimeter on metric spaces},
Ann. Inst. H. Poincar\'e Anal. Non Lin\'eaire 31 (2014), no. 1, 129--154. 

\bibitem{L-Ftype}P. Lahti,
\textit{A new Federer-type characterization of sets of finite perimeter},
Arch. Ration. Mech. Anal. 236 (2020), no. 2, 801--838.

\bibitem{L-Fed}P. Lahti,
\textit{A Federer-style characterization of sets of finite perimeter on metric spaces},
Calc. Var. Partial Differential Equations, October 2017, 56:150.

\bibitem{L-decom}P. Lahti,
\textit{A note on indecomposable sets of finite perimeter},
Adv. Calc. Var. 16 (2023), no. 3, 559--570.

\bibitem{L-FC}P. Lahti,
\textit{A notion of fine continuity for BV functions on metric spaces},
Potential Anal. 46 (2017), no. 2, 279--294.

\bibitem{L-SS}P. Lahti,
\textit{Capacities and 1-strict subsets in metric spaces},
Nonlinear Anal. 192 (2020), 111695, 29 pp.

\bibitem{L-Fedchar}P. Lahti,
\textit{Federer's characterization of sets of finite perimeter in metric spaces},
Anal. PDE 13 (2020), no. 5, 1501--1519.

\bibitem{L-WC}P. Lahti,
\textit{Superminimizers and a weak Cartan property for p=1 in metric spaces},
J. Anal. Math. 140 (2020), no. 1, 55--87.

\bibitem{L-CK}P. Lahti,
\textit{The Choquet and Kellogg properties for the fine topology when $p=1$ in metric spaces},
J. Math. Pures Appl. (9) 126 (2019), 195--213.

\bibitem{Lat}V. Latvala,
\textit{A theorem on fine connectedness},
Potential Anal. 12 (2000), no. 3, 221--232.

\bibitem{MZ}J. Mal\'{y} and W. Ziemer,
\textit{Fine regularity of solutions of elliptic partial differential equations},
Mathematical Surveys and Monographs, 51. American Mathematical Society, Providence, RI, 1997. xiv+291 pp.

\bibitem{M}M. Miranda, Jr.,
\textit{Functions of bounded variation on ``good'' metric spaces},
J. Math. Pures Appl. (9) 82  (2003),  no. 8, 975--1004.

\bibitem{S}N. Shanmugalingam,
\textit{Newtonian spaces: An extension of Sobolev spaces to metric measure spaces},
Rev. Mat. Iberoamericana 16(2) (2000), 243--279.

\end{thebibliography}
\end{document}